\documentclass[12pt]{amsart} 
\usepackage{verbatim, latexsym, amssymb, amsmath,color, mathrsfs}
\usepackage{enumitem}
\usepackage{epsfig}
\usepackage{euscript}
\usepackage{stmaryrd}
\usepackage{dsfont}

\usepackage[dvipsnames]{xcolor}

\usepackage{tikz, pgfplots}
\usetikzlibrary{arrows.meta,positioning}
\usetikzlibrary{positioning}
\usetikzlibrary{matrix,arrows.meta}

\makeatletter
\renewcommand{\tocsection}[3]{%
  \indentlabel{\@ifnotempty{#2}{\bfseries\ignorespaces#1 #2\quad}}\bfseries#3}
\renewcommand{\tocsubsection}[3]{%
  \indentlabel{\@ifnotempty{#2}{\ignorespaces#1 #2\quad}}#3}

\newcommand\@dotsep{4.5}
\def\@tocline#1#2#3#4#5#6#7{\relax
  \ifnum #1>\c@tocdepth 
  \else
    \par \addpenalty\@secpenalty\addvspace{#2}%
    \begingroup \hyphenpenalty\@M
    \@ifempty{#4}{%
      \@tempdima\csname r@tocindent\number#1\endcsname\relax
    }{%
      \@tempdima#4\relax
    }%
    \parindent\z@ \leftskip#3\relax \advance\leftskip\@tempdima\relax
    \rightskip\@pnumwidth plus1em \parfillskip-\@pnumwidth
    #5\leavevmode\hskip-\@tempdima{#6}\nobreak
    \leaders\hbox{$\m@th\mkern \@dotsep mu\hbox{.}\mkern \@dotsep mu$}\hfill
    \nobreak
    \hbox to\@pnumwidth{\@tocpagenum{\ifnum#1=1\bfseries\fi#7}}\par
    \nobreak
    \endgroup
  \fi}
\AtBeginDocument{%
\expandafter\renewcommand\csname r@tocindent0\endcsname{0pt}
}
\def\l@subsection{\@tocline{2}{0pt}{2.5pc}{5pc}{}}
\makeatother

\usepackage{geometry}
\usepackage{hyperref}

\DeclareMathOperator{\Ric}{Ric}

\DeclareMathOperator{\Hess}{Hess}

\DeclareMathOperator{\Id}{Id}

\DeclareMathOperator{\Tr}{Tr}

\DeclareMathOperator{\dvol}{dvol}
\DeclareMathOperator{\Cov}{Cov}
\DeclareMathOperator{\Var}{Var}
\DeclareMathOperator{\supp}{supp}
\DeclareMathOperator{\conv}{conv}

\DeclareMathOperator{\R}{\mathbb{R}}

\DeclareMathOperator{\E}{\textbf{E}}

\newtheorem{thm}{Theorem}[section]

\newtheorem{lem}[thm]{Lemma}

\newtheorem{cor}[thm]{Corollary}

\newtheorem{prob}[thm]{Problem}

\theoremstyle{defn}
\newtheorem{rem}[thm]{Remark}

\begin{document}
\title{Sum of Gaussian vectors and large sets}

\author{Antoine Song}
\address{California Institute of Technology\\ 177 Linde Hall, \#1200 E. California Blvd., Pasadena, CA 91125}
\email{aysong@caltech.edu}

\maketitle

\begin{abstract} 
We prove that the convexity problem of M. Talagrand is equivalent to the subgaussian vector problem: can any centered $1$-subgaussian random vector in $\mathbb{R}^n$ be realized as the sum of a universal number of standard Gaussian vectors?
We introduce methods to study this problem and, using elementary arguments, we settle it for $1$-subgaussian random variables and random vectors with good norm and covariance bounds.
These results already confirm the permutation invariant case of the convexity problem, and give optimal estimates on the largest ellipsoid contained in a sum of large sets in Gaussian spaces.
We also propose a Riemannian version of the convexity problem for spaces with nonnegative Ricci curvature.

\end{abstract}

\section*{Introduction}

One of the questions motivating this paper is the convexity problem of M. Talagrand \cite[Problem 2.3]{Tal95} \cite[Conjecture 2.1]{Tal10} \cite{Tal26} \cite[Problem 54]{Green24} stated below, where $+$ denotes the Minkowski sum and  $\gamma_n$ is the standard Gaussian measure on $\R^n$: 
\begin{prob}[Convexity problem \cite{Tal95,Tal10}] \label{T}
Does there exist a positive integer $q$ such that for any $n\geq 1$ and any closed set $A$ in $\R^n$ with $\gamma_n(A)\geq \frac{2}{3}$, there is a convex body $K$  in $\R^n$ such that
$$\gamma_n(K)\geq \frac{1}{2}\quad \text{and}\quad  K\subset \underbrace{\vphantom{\Big|}A+\cdots+A}_{q\ \text{times}}\quad ?$$
\end{prob}

Consider now an apparently unrelated question.
A standard Gaussian random vector is a random vector in $\R^n$ with probability distribution $\gamma_n$. Given $\kappa>0$, a random vector $Y$ in $\R^n$ is called $\kappa$-subgaussian \cite[Definition 3.4.1, Proposition 2.6.1]{Ver18}
if for any unit vector $v$, we have $\mathbb{P}[|\langle Y,v\rangle |\geq t] \leq 2\exp(-\frac{t^2}{2\kappa^2})$.

\begin{prob}[Subgaussian vector problem]\label{S}
Does there exist an  integer $q>0$ such that for any $n\geq1 $ and any centered $1$-subgaussian random vector $X$ in $\R^n$, there are standard Gaussian random vectors $G_1,...,G_q$  in $\R^n$ with
$$X  =G_1+...+G_q\quad ?$$
\end{prob}
Given standard Gaussian random vectors $G_1,...,G_q$ (independent or not), it is well-known that the sum $G_1+...+G_q$ is centered and $O(q)$-subgaussian \cite[Theorem 1.2]{Bul00} \cite[Exercise 2.42]{Ver18}. Problem \ref{S} asks whether a converse holds. 
It turns out that this problem is crucial for understanding the convexity problem: we will show in Theorem \ref{thm:equivalence} the following:
\begin{thm} \label{thm:equiivv}
$$\text{Problem \ref{T}} \Leftrightarrow \text{Problem \ref{S}}
$$
\end{thm}
This means that a positive answer to one problem implies a positive answer to the other. 
A positive answer would be striking  from both the geometric viewpoint and the probabilistic viewpoint. 
Regardless of the answer, general statements about sums of Gaussian random  vectors are of special interest since,  in principle, they have implications for sums of general centered random vectors with finite second moments via the central limit theorem.
In this paper, we develop new methods for studying \emph{sums of non-independent Gaussian random vectors}, with several geometric consequences. 

\textcolor{RedViolet}{\textbf{Added May 2026:} The first version of this paper appeared in February 2026. In a subsequent paper \cite{HST26} written jointly with Dongming (Merrick) Hua and Stefan Tudose, we solve both problems above. 
Although the results described below in this introduction are now formally subsumed by \cite{HST26}, the present paper played a key role in the genesis of \cite{HST26} both on a psychological level and a technical level. 
Some of the discussions, questions and proofs in this paper remain of interest:
\begin{itemize}
\item The equivalence between Problem \ref{T} and Problem \ref{S} is  a useful translation.
\item In Problem \ref{GC}, we propose a Riemannian version of the convexity problem.
\item The proofs of Theorem \ref{thm:1d}, Theorem \ref{thm:general}, Corollary \ref{cor:permutation}, Corollary \ref{cor:large ellipsoid} are more concrete than their derivation from \cite{HST26}. 
\item  In Subsection \ref{rem:un}, we outline a quick way to get van Handel's strengthening of Talagrand's subgaussian comparison theorem \cite[Corollary 1.2]{Van25}  from Talagrand's classical result \cite[Theorem
2.10.11]{Tal21} and a tensoring trick,  by showing that $1$-subgaussian vectors essentially form the convex hull of standard Gaussian vectors up to a universal scaling. 
\item Lemma \ref{corrrect} gives a sharp bound on the empirical distribution of a sequence of independent random variables whose average distribution is Gaussian, and corrects a gap in \cite{Tal95}. 
\end{itemize}
All the main new ideas in this paper were generated the old-fashioned way, using biological neurons.
}

\subsection*{Sums of Gaussian vectors}
In dimension $1$, M. Talagrand conjectured in \cite[Conjecture 2.7]{Tal95} a statement stronger than Problem \ref{S}: summing three Gaussian random variables should be enough to obtain any sufficiently subgaussian variable. Our first result confirms this conjecture:

\begin{thm} [Three Gaussians]\label{thm:1d}
There is a universal constant  $\kappa>0$ such that given any centered real-valued $\kappa$-subgaussian  random variable $X$, there are three standard Gaussian random variables $G_1,G_2,G_3$ with
$$X=G_1+G_2+G_3.$$
\end{thm}
This result is optimal because there are arbitrarily subgaussian random variables which are not a sum of two standard Gaussian random variables \cite[Proposition 2.6]{Tal95} \cite[Subsection 1.2]{LSS22}. Besides by S. Johnston \cite[Theorem 4.2, Corollary 4.4]{Joh25}, for any integer $k>0$, there are weighted averages of standard Gaussian random variables which are not a weighted average of $k$ standard Gaussian random variables.

We now turn to the high-dimensional case of the subgaussian vector problem. 
A central result in our paper shows that for subgaussian random vectors with good bounds on norm and covariance, the answer to Problem \ref{S} is indeed positive:
\begin{thm}[Norm vs Covariance]\label{thm:general} 
There is a positive integer  $q>0$ such that for any dimension $n\geq 1$, any $\Lambda \geq 1$ and any centered random vector $X$ in $\R^n$ with  $$\|X\|\leq \Lambda \text{ almost surely and}\quad \|\Cov X\|\leq {\Lambda^2 }{e^{-\Lambda^2}},$$ there are standard Gaussian random vectors $G_1,...,G_q$ in $\R^n$ with $$X = G_1+...+G_q.$$ 
\end{thm}
Here, $\|\Cov X\|$ denotes the operator norm of the covariance matrix of $X$. Any centered random vector $X$ in $\R^n$ is $O(1)$-subgaussian as long as $\|\Cov X\|$ is small enough compared to an upper bound on $\|X\|$. The optimal trade-off is exactly as in Theorem  \ref{thm:general}: if $\|X\|\leq \Lambda$ almost surely, then $X$ is $O(1)$-subgaussian when $ \|\Cov X \|\leq {\Lambda^2 }{e^{-\Lambda^2}}$, and the bound ${\Lambda^2 }{e^{-\Lambda^2}}$ is the largest possible up to rescaling $\Lambda$ by a uniform factor. To get a feel for Theorem  \ref{thm:general}, the reader might try to prove that the following centered random vector $X$ is the sum of three Gaussian vectors (see Lemma \ref{bessel} and Subsection \ref{rem:generalize mss?}):
$X$ is uniformly distributed on the $n+1$ vertices of a regular $n$-simplex in $\R^n$, rescaled so that $\|X\| = C_n\sqrt{\log n}$ almost surely, where $C_n>0$ is a well-chosen constant uniformly bounded from above and below.  

\subsection*{Sums of large sets} 
A set $A$ in $\R^n$ is called permutation invariant when $(x_1,...,x_n)$ is in $A$ if and only if $(x_{\sigma(1)},...,x_{\sigma(n)})$ is in $A$ for any permutation $\sigma$ of $\{1,...,n\}$. 
The  permutation invariant case of the convexity problem is already surprisingly difficult. We  settle that case by combining an argument\footnote{\cite[Proposition 2.10]{Tal95} outlines an argument proving Corollary \ref{cor:permutation} based on Theorem \ref{thm:1d}, but it relies on \cite[Lemma 2.11]{Tal95} which is not correct (see Subsection \ref{subsection:permutation}). Fortunately, we can fix this issue by appealing to Theorem \ref{thm:general}.} of M. Talagrand \cite[Proof of Proposition 2.10]{Tal95} with both Theorems \ref{thm:1d} and \ref{thm:general}:

\begin{cor}[Permutation invariant sets] \label{cor:permutation}
There is a universal integer $q>0$ such that if $A$ is a permutation invariant, closed set in $\R^n$ with $\gamma_n(A)\geq \frac{2}{3}$, then there is a  permutation invariant convex body $K$ in $\R^n$ with 
$$\gamma_n(K)\geq\frac{1}{2}
\quad \text{and}\quad  K\subset  \underbrace{\vphantom{\Big|}A+\cdots+A}_{q\ \text{times}}.$$
\end{cor}
The convex body $K$ here is explicit and is made of vectors whose coordinates are subgaussian. The examples of \cite[Proposition 2.6]{Tal95} show that necessarily $q>2$ in Corollary \ref{cor:permutation}. 
In \cite[Theorem 1.3]{Joh25} S. Johnston found, for any integer $k>0$, examples of permutation invariant sets $A\subset \R^n$ with $\gamma_n(A)\geq 2/3$, such that $\{\sum_{j=1}^k \lambda_j a_j;\quad a_j\in A, \, \lambda_j\in [0,1], \, \sum_{j=1}^k \lambda_j =1\}$ does not contain any convex body $K'$ with $\gamma_n(K')\geq 1/2$.

Without symmetry assumptions, Theorem \ref{thm:general} answers an ellipsoid analogue of the convexity problem: how large is the largest ellipsoid contained inside a sum of large sets?
\begin{cor}[Largest Ellipsoid]\label{cor:large ellipsoid}  There exists a universal integer $q>0$ such that  if $A$ is  a closed set in $\R^n$ with $\gamma_n(A)\geq \frac{2}{3}$, then there is an ellipsoid $E$ in $\R^n$ with  
$$\gamma_n(E)\geq\frac{1}{2} \quad \text{and}\quad   \sqrt{\frac{\log n}{n}} \, E\subset \underbrace{\vphantom{\Big|}A+\cdots+A}_{q\ \text{times}}.$$ 
\end{cor}
A classical result of Steinhaus (Lemma \ref{Steinhaus}) implies that, under the assumptions of  Corollary \ref{cor:large ellipsoid}, 
$A+A$ contains a round Euclidean ball $B$ such that $\gamma_n\big(O(\sqrt{n}) B\big)\geq 1/2$. 
Corollary \ref{cor:large ellipsoid} improves this simple estimate only by a modest factor $\sqrt{\log n}$. 
Yet a logarithmic factor is sometimes what separates an easy result from an optimal one. This is famously the case in the solution to the Kadison-Singer conjecture \cite[Theorem 1.4]{Mar15}, which we will crucially use.
Likewise here: up to universal constants, the order of the rescaling factor $\sqrt{{n}^{-1}{\log n}}$ in Corollary \ref{cor:large ellipsoid} is optimal even if $A$ is convex (Subsection \ref{rem:deux}).

\subsection*{\textcolor{RedViolet}{\textbf{Added May 2026:} A Riemannian convexity problem}}

\textcolor{RedViolet}{
Pushed by the desire to better understand the role of the special structure of $\R^n$ and the Gaussian measure in the solution to Problem \ref{T} in \cite{HST26}, and in view of \cite{Joh25}, we propose the following Riemannian version of the problem. Consider a weighted Riemannian manifold $(M^n,g,\mu:=e^{-f} \dvol_g)$ where $(M^n,g)$ is a complete $n$-dimensional Riemannian manifold with convex boundary, $f$ is a smooth real-valued function on $M^n$ and  $\dvol_g$ is the Riemannian volume density on $M^n$. Recall that the Bakry-Emery Ricci curvature of $(M^n,g,\mu)$ (see e.g. \cite{WW09}) is defined as 
$$\Ric_\mu:= \Ric +\Hess f.$$ 
A subset $C\subset (M^n,g)$ is convex if $C$ contains all length minimizing geodesics in $(M^n,g)$ between any two points of $C$, and  the convex hull $\mathrm{Conv}(Y)$ of a subset $Y$ is the smallest convex subset containing $Y$. 
\begin{prob}[Riemannian convexity problem] \label{GC} Is it true that for some universal integer $q>0$, for any weighted Riemannian manifold $(M^n,g,\mu)$ with $\Ric_\mu\geq 0$ and $\mu(M^n)<\infty$, for any closed subset $A\subset M^n$ with $\mu(A)\geq (1-q^{-1})\mu(M^n)$, there is a convex subset $K$ in $(M^n,g)$ with   
$$\mu(K)\geq q^{-n}\mu(M^n) \quad \text{and}\quad  K\subset \bigcup_{x_1,...,x_q\in A}\mathrm{Conv(\{x_1,...,x_q\})}?$$  \end{prob} 
Special cases include positive Ricci curvature manifolds $(M^n,g)$ endowed with the volume measure $\mu$, or Euclidean space $\mathbb{R}^n$ endowed  with a smooth log-concave measure $\mu$.
This problem is a weak generalization of Problem \ref{T}: in fact, the Gaussian space $(\R^n,\gamma_n)$ has nonnegative Bakry-Emery Ricci curvature, and one can check that a positive solution to Problem \ref{T} (confirmed in \cite{HST26}) yields a positive solution to Problem \ref{GC}. The problem can be stated more generally in the context of $\mathrm{RCD}(K,N)$ spaces, see e.g. \cite[Chapter 6]{GP20} \cite{Sturm23}.
}

\subsection*{Background and context}

Problem \ref{T}. Problem \ref{S} and Problem \ref{GC} highlight in a rather striking way the limits of our understanding of \emph{summation}, both in the context of Minkowski sums of sets in $\R^n$ and sums of random vectors with arbitrary couplings. 

On the side of summation of large sets, Problem \ref{T} proposes a potential bridge between the mature theory of convex bodies and general large sets. 
Early results of R. M. Starr \cite{Sta69}, W. R. Emerson and F. P. Greenleaf \cite{Eme69}, and more recent quantitative work of M. Fradelizi, M. Madiman, A. Marsiglietti and A. Zvavitch \cite{Fra16, Fra18} thoroughly examined the convexification effect of taking $k$-fold averages of sets. Unlike what is asked by Problem \ref{T}, the kind of estimates investigated there inevitably depend on the dimension. In \cite{Tal95}, M. Talagrand already observed that the integer $q$ in Problem \ref{T} must be larger than $2$. Based on optimal transport theory, S. Johnston \cite{Joh25} recently constructed counterexamples to a stronger version of Problem \ref{T} where $k$-fold sums are replaced by $k$-fold averages.
There is also a purely combinatorial version of the convexity problem \cite[Problem 3.4]{Tal95} \cite[Conjecture 7.1]{Tal10}, which is a special case of Problem \ref{T} and which is still open. It has lately received more attention due to exciting advances in the theory of thresholds, see \cite{Par24a,Par24b,Pha25} and references therein. 
In another related area, additive combinatorics, the summation of large sets is much better understood than in Gaussian spaces (see however \cite{Bob12, Fra24}). 
In Subsection \ref{rem:deux}, we outline a ``large slice'' estimate which is reminiscent of the Bogolyubov theorem \cite[Theorem 7.8.3]{Zha23}.

On the side of summation of Gaussian random vectors\footnote{
I would like to thank several people who, after a first draft of this paper was written, made me aware of relevant results: Daniel Dadush pointed out \cite{LSS22}, Ramon Van Handel pointed out \cite{Mao19}, Boaz Klartag pointed out \cite{Eld16}, Roman Vershynin pointed out \cite{MGV24}.}, 
not much was known beyond sums of independent Gaussian random vectors.  An important motivation for Problem \ref{S} is that a positive answer would considerably strengthen M. Talagrand's celebrated subgaussian comparison theorem \cite[Theorem 2.10.11]{Tal21} \cite[Corollary 1.2]{Van25}, which can be reformulated as stating that any $1$-subgaussian random vector is an average of standard Gaussian random vectors up to a uniform scaling (see Subsection \ref{rem:un} for our explanation of this fact). 
Our proof of Theorem \ref{thm:general} also suggests a hypothetical subgaussian analogue of the theorem of A. Marcus, D. Spielman and N. Srivastava \cite[Corollary 1.5]{Mar15} (Subsection \ref{rem:generalize mss?}), which would answer\footnote{\textcolor{RedViolet}{\textbf{Added May 2026:} This strategy ended up working, see Subsection \ref{rem:generalize mss?} and \cite[Appendix B]{HST26}.}} Problem \ref{S}. In the context of a scheduling problem in statistics, W. Rhee and M. Talagrand constructed in \cite[Theorem 2]{Rhe92} couplings of three random variables almost uniformly distributed on $[0,1]$ with constant sum. Recent elegant work on that topic  includes the result of Y. P. Liu, A. Sah, M. Sawhney \cite[Lemma 5]{LSS22}
which gives a coupling of two standard Gaussian random variables whose sum is a symmetric random variable taking values in $\{a,0,-a\}$ for some $a>0$, and the result of T.T. Mao, B. Wang, R.D. Wang \cite[Theorem 5]{Mao19}
which characterizes, in terms of convex order, sums of three or more random variables uniformly distributed on $[0,1]$ (see also \cite{Wan15}).
The analogue of \cite[Theorem 5]{Mao19} unfortunately does not hold for sums of standard Gaussian random variables because of \cite[Theorem 4.2, Corollary 4.4]{Joh25}. 
 In higher dimensions, one could show
from arguments of R. Eldan \cite[Theorem 1.4]{Eld16} that centered random vectors with $1$-uniformly log-concave distributions (which are automatically $1$-subgaussian) are sums of two standard Gaussian random vectors. We will recover this fact via simpler methods in Lemma \ref{lem:average of gaussians} (combined with Theorem \ref{caf}). Finally other questions about sums of  Gaussian vectors appear in \cite{MGV24}.

\subsection*{Overview}
The equivalence in Theorem \ref{thm:equiivv} follows from Theorem \ref{thm:equivalence}. The direction ``Problem \ref{T} implies Problem \ref{S}'' is based on a  tensoring trick, a different version of which already appears in \cite[Proposition 2.6]{Tal95}, and the result that given a large convex body, any 1-subgaussian random vector tends to belong to it \cite{Tal21}. The other direction crucially relies on a result of D. Dadush, S. Garg, S. Lovett and A. Nikolov \cite[Theorem 1.2]{Dad19}, that characterizes sets intersecting large convex sets in Gaussian spaces in terms of supports of 1-subgaussian random vectors. 

The main results, Theorem \ref{thm:1d} and Theorem \ref{thm:general}, are proved based on a blend of analytical and geometric methods.
A key tool behind both results is Lemma \ref{lem:average of gaussians}. It states that any contraction of a standard Gaussian random vector is an average of two standard Gaussian random vectors. It is proved via an application of It\^{o}'s formula, similarly to the proof of subgaussianity of contractions of Gaussian random vectors \cite[Theorem 1.5]{Pis16}. 
What makes this simple lemma very handy is the theory of transport maps, which provides many conditions under which a random vector is a contraction of a Gaussian random vector. In particular, when proving Theorem \ref{thm:1d} and Theorem \ref{thm:general}, we will make use of Caffarelli's contraction theorem \cite{Caf00},  a lemma of S. Bobkov \cite[Lemma 3.1]{Bob10}, and a more recent result of D. Mikulincer and Y. Shenfeld \cite{Mik24,Mik23}.

On a technical level, the most challenging step in the proof of Theorem \ref{thm:1d} is Lemma \ref{lem:density bound}. In that lemma,  we show that given a subgaussian variable $S$, there is a carefully chosen coupling with a Gaussian random variable $G$ so that the sum $S+G$ has density 
uniformly comparable above and below to a single shifted Gaussian density. Our choice of coupling  is rather intricate, because naive couplings (like choosing $G$ to be independent of $S$) fail to be useful. Then, the lemma of S. Bobkov \cite[Lemma 3.1]{Bob10} ensures that $S+G$ is, up to a uniform rescaling, a contraction of a standard Gaussian random variable. Combined with Lemma \ref{lem:average of gaussians}, the proof of Theorem \ref{thm:1d} can then be completed.  

The proof of  Theorem \ref{thm:general} is conceptual. By the work of A. Marcus, D. Spielman and N. Srivastava \cite{Mar15}, a general random vector $X$ in $\R^n$ with good bounds on norm and covariance can be decomposed into a mixture of ``simple'' random vectors $X_\theta$, which are uniformly distributed on a small set of vectors (called Bessel sequences)  enjoying sharp quantitative bounds on their norm and covariance. A typical example of a simple random vector is one which is uniformly distributed on the vertices of a regular $n$-simplex, with norm $\sqrt{\log n}$ almost surely.
In Lemma \ref{bessel}, we show that each $X_\theta$ is a sum of standard Gaussian vectors up to a small error. The proof relies  crucially on Lemma \ref{lem:average of gaussians} again, and also uses Caffarelli's theorem  \cite{Caf00}. We deal with the error term with a result of D. Mikulincer and Y. Shenfeld \cite[Theorem 1.3]{Mik23}. This mostly   proves  Theorem \ref{thm:general}.

\subsection*{Acknowledgements}
First, I would like to thank Huy Tuan Pham for introducing me to the convexity problem, and Michel Talagrand for sharing many interesting and thoughtful comments. 
I am grateful to Ramon van Handel for enlightening conversations and for sending  useful references. I  also thank Assaf Naor for suggesting additional motivations for the problem and connected questions. Thanks to an exchange with Samuel Johnston, I detected an issue in an earlier version of the paper. This article benefited from further discussions with Shiqi Song, Boaz Klartag, Yair Shenfeld, Daniel Dadush, Haotian Jiang, Roman Vershynin, Jinyoung Park and Tom Hutchcroft. A. S. was partially supported by NSF grant DMS-2405175 and an Alfred P. Sloan Research Fellowship.

\subsection{Conventions and notations:}\label{sub:convention}

A random vector $X$ 
in $\R^n$ is centered if $\E[X]=0$.
The covariance matrix of $X$ (resp. of its probability distribution $\mu$) is denoted by $\Cov X$ (resp. $\Cov \mu$). We denote the identity matrix in $\R^n$ by $\Id$ or sometimes $I_n$. 
Given random vectors $X$ in $\R^n$ and $Y_1,...,Y_k$ in $\R^n$, such that the $Y_i$'s are defined on the same probability space, by abuse of notation, we often write $X=Y_1+...+Y_k$ when $X$ and $Y_1+...+Y_k$ have the same probability distribution. We say that $X$ is the sum of $k$ standard Gaussian vectors in $\R^n$ 
if there is a coupling of $k$ random vectors  $Y_1,...,Y_k\sim \mathcal{N}(0,I_n)$ such that the sum has same probability distribution as $X$. 
Given a random vector $X$ we may sometimes define another random vector $Y$ with respect to $X$ and consider $X+Y$, and $Y$ is implicitly assumed to be defined on the same probability space as $X$ (we might need to enlarge it).
For any integer $q>0$, define 
 \begin{equation}\label{gaussian sum notation}
 \begin{split}
\Sigma^q\mathcal{G}(\R^n):=\{&\text{random vectors $X$ in $\R^n$ such that $X=G_1+...+G_q$}\\
&\text{for some random vectors $G_1,...,G_q\sim \mathcal{N}(0,I_n)$}\}.
\end{split}
\end{equation}
Given a vector $x\in \R^n$, $\|x\|$ denotes its Euclidean norm. Given an $n$-by-$n$ matrix $M$, $\|M\|$ denotes its operator norm.
A set $A'\subset \R^n$ is symmetric if $A'=-A'$. Given a set $A$ in $\R^n$ and an integer $k>0$, we denote its $k$-fold Minkowski sum by
 \begin{equation}\label{sum notation}
 A_{(k)}:=A+...+A \quad \text{($k$ times)}.
 \end{equation}
Given a function $f:\R\to \R$, we will use the usual notation $O(f)$ (resp. $\Theta(f)$) to denote a function $g:\R\to \R$ such that $g(x)\leq Cf(x)$ (resp. $cf(x)\leq g(x)\leq Cf(x)$) for some constants $c,C>0$ independent of $x$. The floor function is denoted by $\lfloor x\rfloor:=$ largest integer $\leq x$.

\section{Equivalence}\label{section:equivalence}

In this section, which can mostly be read independently of the next sections, we explain the equivalence between Problem \ref{T}, the original formulation\footnote{The original problem  differs slightly from Problem \ref{T} as it assumes sets to be ``balanced''.} of Problem \ref{T} in \cite[Problem 2.3]{Tal95} \cite[Conjecture 2.1]{Tal10}, and Problem \ref{S}.  
We first define certain numbers $q_{S,n}$, $q_{C,n}$ and $q_{C,n}'$. By convention, the smallest element of an empty subset of the integers is defined to be $\infty$. 
A set $A\subset \R^n$ is called \emph{balanced} if it is symmetric and star-shaped, namely for any $x\in A$ and $\lambda\in [-1,1]$, $\lambda x\in A$. Recall also the notations $\sum^q\mathcal{G}(\R^n)$ in (\ref{gaussian sum notation}) and $A_{(k)}$ in  (\ref{sum notation}).

\begin{itemize}

\item $q_{C,n}:=$  smallest positive integer $q$ such that for any closed set $A$ in $\R^n$ with $\gamma_n(A)\geq \frac{2}{3}$, $A_{(q)}$ contains a convex body $K$ with $\gamma_n(K)\geq \frac{1}{2}$.
\vspace{1em}

\item $q_{C,n}':=$ smallest positive integer $q$ such that for any balanced closed set $A$ in $\R^n$ with $\gamma_n(A)\geq \frac{2}{3}$, $A_{(q)}$ contains a convex body $K$ with $\gamma_n(K)\geq \frac{1}{2}$.
\vspace{1em}

\item $q_{S,n}:=$ smallest positive integer $q$ such that for any centered $1$-subgaussian random vector $X$ in $\R^n$, we have $X\in \Sigma^{q'}\mathcal{G}(\R^n)$ for some integer $q'\leq q$.

\end{itemize}

The next result implies Theorem \ref{thm:equiivv}:

\begin{thm}[Equivalence] \label{thm:equivalence}
The following are equivalent: as $n\to \infty$,
\begin{equation*}
\text{$(i)$ $q_{C,n} =O(1)$},\quad \text{$(ii)$ $q_{C,n}' =O(1)$},\quad \text{$(iii)$  $q_{S,n} =O(1)$}.
\end{equation*}
\end{thm}

We will use the following known extension of Steinhaus theorem\footnote{The classical Steinhaus theorem states that for $\varepsilon>0$ small enough, and some $a_\varepsilon>0$, if a subset $A\subset [-1,1]$ has Lebesgue measure at least $2-\varepsilon$, then $A+A$ contains the interval $[-a_\varepsilon,a_\varepsilon]$. }:

\begin{lem}[Steinhaus Lemma]\label{Steinhaus}
There is a $\delta>0$
such that for any dimension $n\geq 1$, for any closed set $A$ in $\R^n$ with $\gamma_n(A)\geq \frac{2}{3}$,   
$$B(0,\delta)\subset A+A$$
where $B(0,\delta)$ is the closed Euclidean ball of radius $\delta$ centered at $0$. In particular, given any $D>0$, there is an integer $p=O(D)$ such that
$$B(0,D)\subset A_{(p)}.$$
\end{lem}

\begin{proof}
Fix $\epsilon>0$. If $\delta>0$ is small enough depending only on $\epsilon$, then for any $n\geq 1$ and any $v\in \R^n$ of norm $\delta$, the Gaussian measure of the set of points $x\in\R^n$ with $|\|x-v\|^2 - \|x\|^2| = |-2\langle x,v\rangle  +\|v\|^2| \leq \epsilon$ is at least $1-\epsilon$. Hence, if $\delta$ is small enough, then for any $n\geq 1$, any closed set $A\subset \R^n$ with $\gamma_n(A)\geq \frac{2}{3}$ and any $v\in \R^n$ with $\|v\|\leq \delta$,
\begin{align*}
\gamma_n(A-v) &= \frac{1}{(\sqrt{2\pi})^n} \int_{A} \exp(-\frac{\|x-v\|^2}{2}) dx \\
& >\frac{9}{10}\frac{1}{(\sqrt{2\pi})^n} \int_{A} \exp(-\frac{\|x\|^2}{2}) dx   =\frac{9}{10} \gamma_n(A)  > 1-\gamma_n(A).
\end{align*}
This implies that $(A-v)\cap (-A) \neq \varnothing$, which gives the lemma.
\end{proof}

From the standard Gaussian isoperimetric inequality \cite[Theorem 2.1]{Lat03}, we have:
\begin{lem}[Volume of neighborhood]\label{cor:vol sum}
Given any dimension $n\geq 1$, any closed set $S$ in $\R^n$ with $\gamma_n(S)\geq \frac{1}{2}$, and any $D>0$,
$$\gamma_n(S+B(0,D)) \geq 1-2\exp(-\frac{D^2}{2}).$$
\end{lem}

We will need the following result of Talagrand \cite{Tal21}, see \cite[Lemma 1.3]{Dad19}\footnote{The original statement was for symmetric convex bodies. Our extension simply follows from the fact that if $K_1$ is a convex body with $\gamma_n(K_1)>9/10$, then $K_1\cap (-K_1)$ is a symmetric convex body with Gaussian measure at least $\frac{1}{2}$.}:
\begin{thm}[\cite{Tal21}]\label{tal}
Let $Y$ be a centered $1$-subgaussian random vector in $\R^n$. There is a universal integer $p'\geq 1$ such that for any convex body $K\subset \R^n$ with $\gamma_n(K)\geq \frac{9}{10}$, we have  $\mathbb{P}[\frac{1}{p'}Y\in K]\geq \frac{1}{2}$.
\end{thm}

Another ingredient will be the following corollary of a result of Dadush-Garg-Lovett-Nikolov \cite[Theorem 1.2]{Dad19}\footnote{In the original statement, $T:=\R^n\setminus S$ is assumed to be finite and $X$ does not satisfy $X\sim -X$, but the result is easily extended by compactness, and using $S=-S$ and  considering a mixture of $X$ and $-X$.}:
\begin{thm}[\cite{Dad19}]\label{dgln}
Let $S\subset \R^n$ be an open symmetric set. If $S$ contains no symmetric convex body $K$ with $\gamma_n(K)\geq \frac{1}{2}$, then there is a $1$-subgaussian random vector $X$ in $\R^n$ such that $X\sim -X$ and supported on $\R^n\setminus \frac{1}{p}S$ where $p>0$ is a universal integer. 
\end{thm}

\begin{proof}[Proof of Theorem \ref{thm:equivalence}]
Obviously, we have  $q_{C,n}\geq q_{C,n}'$, so it suffices to show $(ii)\implies (iii)$ and $(iii)\implies (i)$. 
We will use the usual notation  for the 1-Wasserstein distance between two random vectors $X,Y$ in $\R^n$:
$$W_1(X,Y) := \inf_{\substack{X'\sim P_X\\ Y' \sim P_Y}}\E\|X'-Y'\|$$
where the infimum is taken over all couplings of the respective probability distributions $P_X$ and $P_Y$ of $X$ and $Y$.

Assume that $(ii)$ holds, namely that ${q_{C,n}'}\leq q$ for some integer $q>0$ independent of $n$.
Let $X$ be a centered $1$-subgaussian random vector in $\R^n$. Fix some $\epsilon>0$.
Given an integer $M>0$ and $x=(x_1,...,x_M)\in \R^n\times...\times\R^n = \R^{nM}$, let $V_x$ be the random vector in $\R^{n}$ whose distribution is $\frac{1}{M}\sum_{i=1}^M \delta_{x_i}$. Set 
$$\mathfrak{X}_{n,M} :=\{ x\in \R^{nM};\quad W_1(V_x, X) \leq  \epsilon\}.$$ 
Let $G\sim \mathcal{N}(0,I_n)$. 
Set 
$$\mathfrak{G}_{n,M}:=\{x\in \R^{nM};\quad W_1(V_x, \lambda G) \leq  \epsilon \text{ for some $\lambda\in [-1,1]$}\}.$$ 
The set $\mathfrak{G}_{n,M}$ is clearly balanced and closed.
By the Glivenko-Cantelli theorem \cite{Var58}, $\gamma_{nM}(\mathfrak{G}_{n,M}) \geq 2/3$ for all $M$ large enough.
By the inequality $q_{C,n}'\leq q$, by Lemmas \ref{Steinhaus} and \ref{cor:vol sum}, and since $\mathfrak{G}_{n,M}$ is symmetric and contains $0$, there exist a universal integer $c_1>0$ and a convex body $K_{n,M}$ in $\R^{nM}$  such that
\begin{equation}\label{qcn'}
\gamma_{nM}(K_{n,M}) \geq \frac{9}{10}\quad \text{and}\quad  K_{n,M}\subset \mathfrak{G}_{n,M}+...+\mathfrak{G}_{n,M} \quad (\text{$c_1q$ times}).
\end{equation}
Let $X_1,...,X_M$ be i.i.d. copies of $X$. Then the centered random vector
$\hat{X}:=(X_1,...,X_M)$ in $\R^{nM}$ is still 1-subgaussian (use characterization $(iv)$ in \cite[Proposition 2.6.1]{Ver18}), and so by Theorem \ref{tal}, $\frac{1}{p'}\hat{X}\in K_{n,M}$
with probability at least $1/2$ for some universal integer $p'>0$. Moreover, $\hat{X}\in\mathfrak{X}_{n,M} $ with probability at least $9/10$ whenever $M$ is large enough, by the Glivenko-Cantelli theorem \cite{Var58} again.
So by the union bound, there exists a point $x\in \frac{1}{p'}\mathfrak{X}_{n,M}\cap K_{n,M}  \subset \R^{nM}.$ 
By (\ref{qcn'}), there are  some $g_1,...,g_{c_1q} \in \mathfrak{G}_{n,M}$ such that
$$x = g_1+...+g_{c_1q}.$$
By definitions of $\mathfrak{X}_{n,M}$ and $\mathfrak{G}_{n,M}$, we conclude that for all $M$ large enough, for some $\lambda_1,...,\lambda_{c_1q}\in [-1,1]$ and some standard Gaussian vectors $G_1,...,G_{c_1q}$ in $\R^n$, 
$$W_1(\frac{1}{p'}X,\lambda_1G_1+...+\lambda_{c_1q}G_{c_1q}) \leq (1+c_1q)\epsilon.$$
After letting $M\to \infty$ and $\epsilon\to 0$, we obtain by Prokhorov's compactness theorem:
$$\frac{1}{p'}X =\lambda_{0,1} G_{0,1}+...+\lambda_{0,c_1q}G_{0,c_1q}$$
for  some $\lambda_{0,1},...,\lambda_{0,c_1q}\in [-1,1]$ and some standard Gaussian random vectors $G_{0,1}, ...,G_{0,c_1q}$ in $\R^n$. 
We will see in Corollary \ref{cor:rescale} that  each $\lambda_{0,j}G_{0,j}$ is the sum of two standard Gaussian random  vectors in $\R^n$. We conclude that $q_{S,n}\leq 2p'c_1q$ and that $(iii)$ holds as wanted.

Now assume that $(iii)$ holds, namely that ${q_{S,n}}>0$ is finite and bounded independently of $n$. 
Let $A$ be a closed set in $\R^n$ such that $\gamma_n(A)\geq 2/3$. By Lemmas \ref{Steinhaus} and \ref{cor:vol sum}, for a universal integer $c_2> 2$ large enough, $A_{(c_2)}\cap (-A_{(c_2)})$ is symmetric, and its interior has Gaussian measure at least $2/3$.
So we see that we only need to show the conclusion of $(i)$ when $A$ is both open and symmetric, and $\gamma_n(A)\geq 2/3$.
Let $\delta>0$ be the universal radius given by Lemma \ref{Steinhaus}. Combining Lemmas \ref{Steinhaus} and \ref{cor:vol sum}, there is an integer $k_1>100$ depending only on $q_{S,n}'$ and $\delta$ such that 
\begin{equation} \label{gew10q}
\gamma_n(A_{(k_1)})\geq 1-\frac{1}{10q_{S,n}'}.
\end{equation}
Let $p>0$ be the universal integer given by Theorem \ref{dgln}.
Set
 \begin{equation}\label{k_0}
 k_0:=pk_1q'.
  \end{equation} 
Suppose towards a contradiction that 
$A_{(k_0)}$, which is open and symmetric by our reduction,
does not contain a convex body of Gaussian measure at least $1/2$.  
By Theorem \ref{dgln}, this assumption implies the existence of a centered $1$-subgaussian random vector $X$ in $\R^n$ with $X\sim -X$ and with support 
 \begin{equation}\label{x in} 
\supp X \subset \R^n \setminus \frac{1}{p}A_{(k_0)} \underset{(\ref{k_0})}{\subset}  \R^n \setminus A_{(k_1q')}.
 \end{equation} 
Since ${q_{S,n}}<\infty$, there are standard Gaussian random vectors $G_1,...,G_{q'}$  for some integer $q'\leq q_{S,n}$, such that 
 \begin{equation} \label{xggr}
 X= G_1+...+G_{q'}.
\end{equation} 
Then by (\ref{gew10q}) and the union bound, with positive probability, $G_j\in A_{(k_1)}$ for all $j=1,...,q'$. 
Thus, by (\ref{xggr}), there is $x\in \supp X$,  and there are $g_1,...,g_{q'}\in A_{(k_1)}$, such that
 $ x=g_1+...+g_{q_{S,n}'}$, so in particular  $x\in A_{(k_1q')},$
which is a contradiction with (\ref{x in}). This shows that $(i)$ actually holds with 
$$q_{C,n}\leq k_0=p k_1q'.$$

\end{proof}


\section{How expressive are sums of Gaussian vectors?}\label{section:expressive}

\subsection{Basic properties of sums of  Gaussian  vectors}
In this subsection, we develop some general tools for studying sums of non-independent Gaussian random vectors. In preparation for the main lemma of this subsection, we note the following:
\begin{lem}[Linear image]\label{lem:decomposition}
If $G$ is a standard Gaussian random vector in $\R^n$, then for any linear operator $A$ with $\|A\|\leq 1$,
there are two standard Gaussian random vectors $X,Y$ in $\R^n$ such that 
 $$A(G)=\frac{X+Y}{2}.$$
\end{lem}

\begin{proof}
The singular value decomposition of $A$, viewed as an $n$-by-$n$ matrix, yields $A=USV^T$ where $U,V$ are orthogonal matrices and $S$ is diagonal with eigenvalues in $[0,1]$, since $\|A\|\leq 1$. 
By rotational symmetry of standard Gaussian random vectors, we can assume without loss of generality that $A=S$. 
Note that $A(G)$ is then given by a random vector of the form $(Z_1,...,Z_n)$ where $Z_i$ are independent real-valued Gaussian random variables with variance at most $1$. 
Thus, it becomes clear that to prove the lemma, it is enough to show that any real-valued Gaussian random variable with variance at most $1$ is the average of two standard  Gaussian random variables. 
Let $Z'$ be a real-valued Gaussian random variable with variance $\sigma^2\leq 1$. Let $X_1$ be a standard Gaussian real-valued random variable, which we can decompose as $X_1=X_{\sigma^2}+X_{1-\sigma^2}$
where $X_{\sigma^2},X_{1-\sigma^2}$ are independent Gaussian random variables centered at $0$ and with variances $\sigma^2$, $1-\sigma^2$ respectively.
Now, consider the random variable  $Y_1:= X_{\sigma^2}-X_{1-\sigma^2}.$
Then $Y_1$ is, like $X_1$, a standard Gaussian random variable. But observe that $\frac{1}{2}(X_1+Y_1) = X_{\sigma^2}$ has the same distribution as $Z'$. That finishes the proof.

\end{proof}

Recall the notation $\sum^q\mathcal{G}(\R^n)$ for $q$-fold Gaussian sums defined in (\ref{gaussian sum notation}).
\begin{cor}[Linear image II]\label{cor:linear bis}
Let $X\in \sum^q\mathcal{G}(\R^n)$ for some integer $q>0$.
Then, for any linear map $F:\R^n\to \R^n$ such that $\|F\|\leq 1$, $F(X)\in \sum^{2q}\mathcal{G}(\R^n)$. 
\end{cor}
\begin{proof}
It is enough to 
note that if $G\sim \mathcal{N}(0,I_n)$ and if $A:\R^n\to \R^n$ is a linear map with $\|A\|\leq 1$, then there are standard Gaussian random vectors $X,Y$ such that $A(G) = 2 \frac{A(G)}{2} =  2\frac{X+Y}{2} = X+Y$ by Lemma \ref{lem:decomposition} applied to $\frac{A}{2}$.
\end{proof}

\begin{cor}[Rescaling]\label{cor:rescale}
Let $X\in \sum^q\mathcal{G}(\R^n)$ for some integer $q>0$. For any $\tau> 0$, $\tau X\in \sum^{( \lfloor \tau \rfloor +2) q}\mathcal{G}(\R^n)$.
\end{cor}
\begin{proof}
Use Corollary \ref{cor:linear bis} and the fact that
$\tau X = \lfloor \tau \rfloor X
+ (\tau- \lfloor \tau \rfloor)X$.
\end{proof}

Next, we state the main lemma of this subsection. It will be especially useful when combined with the contraction theorems of Caffarelli (Theorem \ref{caf}) and Bobkov (Theorem \ref{thm:bobkov}).

\begin{lem}[Lipschitz image] \label{lem:average of gaussians}
Let $\Psi:\R^n\to \R^n$ be a Lipschitz map with Lipschitz constant at most $C_{\mathrm{Lip}}>0$, and let $G$ be a standard Gaussian random vector in $\R^n$.
Then there are two standard Gaussian random vectors $X,Y$ in $\R^n$ such that 
 $$\Psi(G)-\E[\Psi(G)]=C_{\mathrm{Lip}}\frac{X+Y}{2}.$$

\end{lem}

\begin{proof}
Our proof relies on  It\^{o}'s formula,  in a way inspired by the proof of \cite[Theorem 1.5]{Pis16}.
By rescaling $\Psi$ by a factor $\frac{1}{C_{\mathrm{Lip}}}$, it is enough to show the statement when $\Psi$ is 1-Lipschitz.
Let $\{B_t\}$ be the standard Brownian motion on $\R^n$ starting at $0$. Note that $\Psi(G)-\E[\Psi(G)]$ has same probability distribution as $\Psi(B_1)-\E[\Psi(B_1)].$
By It\^{o}'s formula  \cite[Theorem 3.3]{Rev99}, we have 
$$\Psi(B_1) -\E[\Psi(B_1)] = \int_0^1 \nabla(P_{1-t}\Psi)(B_t).dB_t.$$
Here, $\{P_t\}_{t\geq 0}$ is the heat semigroup, and if $\Psi$ is the vector-valued map $(\Psi_1,...,\Psi_n)$, then $P_{1-t}\Psi$ denotes the vector-valued map $(P_{1-t}\Psi_1,...,P_{1-t}\Psi_n)$ and $\nabla(P_{1-t}\Psi)(x)$ is the linear map whose matrix has $i$-th row equal to $(\frac{\partial}{\partial x_1} P_{1-t}\Psi_i(x), ..., \frac{\partial}{\partial x_n} P_{1-t}\Psi_i(x))$.
Since $\Psi$ is 1-Lipschitz, $P_{1-t}\Psi$ is also 1-Lipschitz and so 
$\|\nabla(P_{1-t}\Psi)\| \leq 1.$
This means by  \cite[Proposition 2.13]{Rev99} that the random variable $\Psi(G)-\E[\Psi(G)]$ is the limit in probability of random variables of the following form:
$$\sum_{j=1}^N A_j(G_{N,j})$$
where $G_{N,j}\sim \mathcal{N}(0,\frac{1}{N}I_n)$ are independent and for each $j\geq 1$, $A_j$ is a certain random linear operator with operator norm at most 1 depending only on $G_{N,1},...,G_{N,j-1}$.
By Lemma \ref{lem:decomposition}, for each $j=1,...,N$, conditioned on $A_j$ (which only depends on $G_{N,k}$ where $k=1,...,j-1$),
$$A_j(G_{N,j}) = \frac{X_{N,j}+Y_{N,j}}{2}$$
where $X_{N,j}, Y_{N,j} \sim \mathcal{N}(0,\frac{1}{N}I_n)$. 
Since the latter holds no matter what the values of $G_{N,k}$ are for $k=1,...,j-1$, this gives well-defined random vectors $X_{N,1},...,X_{N,N}\sim \mathcal{N}(0,\frac{1}{N}I_n)$ which are independent, and similarly random vectors $Y_{N,1},...,Y_{N,N}\sim \mathcal{N}(0,\frac{1}{N}I_n)$ which are independent.
Summing in $j$, we get 
$$\sum_{j=1}^N A_j(G_{N,j}) = \frac{X_N+Y_N}{2}$$
where $X_N:=\sum_{j=1}^N X_{N,j}$ and $Y_N:=\sum_{j=1}^N Y_{N,j}$
satisfy $X_N,Y_N \sim \mathcal{N}(0,I_n)$.
By Prokhorov's compactness theorem, since $\Psi(G)-\E[\Psi(G)]$ is the limit in distribution of $\sum_{j=1}^N A_j(G_{N,j})$, $\Psi(G)-\E[\Psi(G)]\sim \frac{X+Y}{2}$
where $X,Y \sim \mathcal{N}(0,I_n)$.

\end{proof}

\begin{cor}[Bounded support]\label{cor:bdd support}
For any centered random vector $X$ in $\R^n$ with $\|X\|\leq 1$ almost surely, we have $X\in \sum^5\mathcal{G}(\R^n)$.
\end{cor}
\begin{proof}
Let $G$ be a standard Gaussian random vector in $\R^n$ independent of $X$. By \cite[Theorem 1.3]{Mik24} (see also \cite[Theorem 2]{Mik23}) and by using that $-G_0\sim \mathcal{N}(0,1)$ whenever $G_0\sim \mathcal{N}(0,1)$, we obtain $X-G=F(G')$ where $F:\R^n\to \R^n$ is $\frac{\sqrt{e^2-1}}{\sqrt{2}}$-Lipschitz and $G'$ is a standard Gaussian random vector. Lemma \ref{lem:average of gaussians} implies that $X-G = \frac{\sqrt{e^2-1}}{2\sqrt{2}} (G_1+G_2)$ where $G_1,G_2$ are   standard Gaussian random vectors. We conclude with Corollary \ref{cor:rescale} and the fact that $\frac{\sqrt{e^2-1}}{2\sqrt{2}} \leq 1$.   
\end{proof}

We record without proof an elementary fact:
\begin{lem}[Local-to-global]\label{lem:local global}
Let $q>0$ be an integer. Consider a random vector $X$ in $\R^n$ and a discrete random variable $\mu$.
If for any $t$ such that $\mathbb{P}[\{\mu=t\}]>0$, $X$ conditioned on $\{\mu=t\}$  belongs to $\sum^q\mathcal{G}(\R^n)$, then $X\in \sum^q\mathcal{G}(\R^n)$.
\end{lem}

\begin{cor}[Bounded support II]\label{cor:bdd support II}
Let $q>0$, $C_0>0$ be two  integers. 
Consider a  random vector $X$ in $\R^n$ and a discrete random variable $\mu$. Given $t$ such that $\mathbb{P}[\{\mu=t\}]>0$, let $Z_t$ be $X$ conditioned on $\{\mu=t\}$.
If $Z_t-\E[Z_t]\in \sum^q\mathcal{G}(\R^n)$ and  $\|\E[Z_t]\|\leq C_0$ for all $t$, then $X-\E[X]\in \sum^{q+10C_0}\mathcal{G}(\R^n)$.
\end{cor}
\begin{proof}
Consider the indicator function $\mathbf{1}_{\{\mu =t\}}$.
We have
$$X-\E[X]= \sum_{t} \mathbf{1}_{\{\mu =t\}} (Z_t-\E[Z_t]) +\sum_{t} \mathbf{1}_{\{\mu =t\}} \E[Z_t] -\E[X].$$ By assumption on $Z_t$ and Lemma \ref{lem:local global},  $\sum_{t} \mathbf{1}_{\{\mu =t\}} (Z_t-\E[Z_t]) \in \sum^q\mathcal{G}(\R^n)$. By Corollary \ref{cor:bdd support}, since $\|\E[Z_t]-\E[X]\|\leq 2C_0$, we have $\sum_{t} \mathbf{1}_{\{\mu =t\}} \E[Z_t] -\E[X] \in \sum^{10C_0}\mathcal{G}(\R^n)$. 
\end{proof}

The next lemma says that a sum of Gaussian random vectors with small covariance can be rewritten as a sum of standard Gaussian random vectors:

\begin{lem}[Normalization] \label{lem:normalization}
For any integer $q\geq 3$, any standard Gaussian vectors $G_1,...,G_q$ in $\R^n$ and any  $\tau_1,...,\tau_q\in (0,\frac{1}{\sqrt{2}})$, 
$$\sum_{i=1}^q\tau_iG_i \in  \text{${\sum}^q\mathcal{G}(\R^n)$}.$$
\end{lem}
\begin{proof}
Note that, since a sum of $q\geq 3$ terms can be decomposed into a sum of terms which are sums of 3 or 4 or 5 terms,
the lemma holds if it holds for $q\in \{3,4,5\}$.  
Recall that if $G_a\sim \mathcal{N}(0,aI_n)$, $G_b\sim \mathcal{N}(0,bI_n)$ are independent then $G_a+G_b\sim \mathcal{N}(0,(a+b)I_n)$. 
Thus, to prove the lemma, it suffices to show that  for any $q\in \{3,4,5\}$ and any $\tau_i\in (0,\frac{1}{\sqrt{2}})$ where $i\in \{1,...,q\}$, there are random vectors $W_i\sim \mathcal{N}(0,(1-\tau_i^2)I_n)$ such that $\sum_{i=1}^q W_i =0$. 
We explain the case $q=3$ only, since the other cases are similar (and since we will only need that case later). Fix $a:= \frac{1}{2}((1-\tau_1^2)+(1-\tau_2^2) - (1-\tau_3^2))$. Note that  $a\geq 0$, $(1-\tau_1^2)-a\geq 0$  and $(1-\tau_2^2)-a\geq 0$ since $\tau_i\leq \frac{1}{\sqrt{2}}$ for $i\in \{1,2,3\}$. Consider $W_1=W_1'+W_1''$, and $W_2=W_2'+W_2''$ where for $i\in \{1,2\}$, $W_i'$ and $W_i''$ are independent,  $W_i'\sim \mathcal{N}(0,aI_n)$ and $W_i''\sim \mathcal{N}(0,((1-\tau_i^2)-a) I_n)$, $W''_1$ and $W''_2$ are independent, and  moreover $W_1'=-W_2'$. Then $W_1\sim \mathcal{N}(0,(1-\tau_1^2) I_n)$, $W_2\sim \mathcal{N}(0,(1-\tau_2^2) I_n)$, $W_1+W_2 = W_1''+W_2''\sim \mathcal{N}(0,(1-\tau_3^2) I_n)$. Set $W_3:=-(W_1+W_2)$ so that $W_3 \sim \mathcal{N}(0,(1-\tau_3^2) I_n)$, and the case $q=3$ is checked. 
 
\end{proof}



\subsection{Real-valued subgaussian random variables}

Consider the standard normal density $\varphi:\R\to \R$, defined as
\begin{equation}\label{varphi}
\varphi(x):= \frac{1}{\sqrt{2\pi}} \exp(-\frac{x^2}{2}).
\end{equation}
The following lemma is technically a key piece of the proof of Theorem \ref{thm:1d}.
\begin{lem}[Density bounds]\label{lem:density bound} 
For some universal $\kappa\in (0,1)$, if $S$ is a real-valued $\kappa$-subgaussian variable, then there is a standard Gaussian random variable $G$ such that $S+G$ has distribution $f(x)dx$ where  $f:\R\to \R$ is a piecewise continuous function that satisfies
$$C^{-1} \varphi(x-y_0)\leq f(x) \leq C\varphi(x-y_0)$$
for some $y_0\in [-1,1]$ depending on $S$ and for some  universal constant $C\geq 1$.  
\end{lem}
\begin{proof}
Let $\kappa\in (0,\frac{1}{2})$, which we will fix in the proof. For simplicity, we assume that $S$ is a discrete variable with density $\sum_{i=1}^I p_i\delta_{x_i}$ where $I>0$ is an integer, $p_i\in (0,1)$, $\sum_{i=1}^Ip_i =1$, and $x_i\in \R$. The general case can then be deduced by approximating the variable $S$ in the weak topology by such discrete variables, and use compactness.

Consider some $y_0\in [-1,1]$ and some $\nu>0$ for the moment.
Given $j=1,...,I$,
if $x_j\notin [-1,1]$, set 
$$g_{j,0}(x):= 0,\quad g_{j,1}(x):= \varphi(x-x_j)$$
and if $x_j\in [-1,1]$, set
$$g_{j,0}(x):= \nu \varphi(x-y_0)\mathds{1}_{j,y_0}(x),\quad g_{j,1}(x):= \varphi(x-x_j) - g_{j,0}(x)$$
where $\mathds{1}_{j,y_0}$ denotes the function equal to $\mathds{1}_{[y_0,\infty)}$ if $y_0< x_j$ and equal to $\mathds{1}_{(-\infty,y_0]}$ if $x_j\leq y_0$.
If $\nu\in (0,1/2)$ is smaller than some universal positive constant, then $g_{j,1}(x)>0$ for all $j=1,...,I$, all $x\in \R$ and all choices of $y_0\in [-1,1]$.
If $x_j\notin [-1,1]$, set 
$$\alpha_{j}:=0, \quad \beta_{j,-} = \beta_{j,+}:=\frac{1}{2}$$
and if $x_j\in [-1,1]$, set
$$\alpha_{j}:= \int_{\R } g_{j,0}(x) dx, 
\quad \beta_{j,-}:= \int_{\R}g_{j,1}(x) \mathds{1}_{\{x<x_j\}}(x) dx,
\quad \beta_{j,+}:= \int_{\R}g_{j,1}(x)\mathds{1}_{\{x\geq x_j\}}(x) dx.$$
Note that  for any $j=1,...,I$, 
\begin{align}\label{varphixxj}
g_{j,0}(x) + g_{j,1}(x) = \varphi(x-x_j) \quad \text{for all $x\in \R$},
\end{align}
\begin{align}\label{abb}
\alpha_{j}+ \beta_{j,-}+ \beta_{j,+}=1,
\end{align}
\begin{align}\label{b>c}
\beta_{j,-}>c\quad \text{and}\quad  \beta_{j,+}>c\quad \text{for a constant $c>0$ depending only on $\nu$}.
\end{align}

Given $y_0\in [-1,1]$, renumber the $x_i$'s so that $\{x_1,.,,,x_{I_{y_0}}\} = \{x_i\}_{i=1}^I \cap (-\infty,y_0]$, $\{x_{I_{y_0}+1},.,,,x_{I}\} = \{x_i\}_{i=1}^I \cap (y_0,\infty)$. As a convention, $I_{y_0} := 0$ when the first set is empty, and $I_{y_0}:=I$ when the second one is empty. A corresponding convention is assumed when taking sums indexed by $1\leq i\leq I_{y_0}$  or $I_{y_0}+1\leq i\leq I$.
For $\kappa>0$ smaller than some universal positive constant, we claim that there is a choice of $y_0\in [-1,1]$ such that 
\begin{enumerate} [label=(\alph*)]
\item either $y_0\notin  \{x_i\}_{i=1}^I$ and 
\begin{equation}\label{casea}
\sum_{i=1}^{I} p_i \beta_{i,+}  = \sum_{i=1}^{I_{y_0}}p_i(1-\alpha_i)\quad  \text{and}\quad   \sum_{i=1}^{I} p_i \beta_{i,-} = \sum_{i=I_{y_0}+1}^{I}p_i (1-\alpha_i).
\end{equation}
\item or $y_0 = x_{I_{y_0}} \in  \{x_i\}_{i=1}^I$, and for some $p'\in (0,p_{I_{y_0}})$, 
$$\sum_{i=1}^{I} p_i \beta_{i,+}  = \sum_{i=1}^{I_{y_0}-1}p_i(1-\alpha_i) + p'(1-\alpha_{I_{y_0}}) \quad \text{and}\quad  \sum_{i=1}^{I} p_i \beta_{i,-} = \sum_{i=I_{y_0}+1}^{I}p_i (1-\alpha_i) +(p_{I_{y_0}}-p')(1-\alpha_{I_{y_0}}).$$
\end{enumerate}
As we explain next, this follows from a continuity argument. First, note that by (\ref{abb}) we have 
$$\sum_{i=1}^{I} p_i \beta_{i,+} +  \sum_{i=1}^{I} p_i \beta_{i,-} = \sum_{i=1}^{I_{y_0}}p_i(1-\alpha_i)+\sum_{i=I_{y_0}+1}^{I}p_i (1-\alpha_i),$$
so if one equality holds in Case (a) or Case (b) then the other one automatically follows. 
As $y_0$ varies from $-1$ to $1$, for each $j=1,...,I$, it is easy to check that $\alpha_j$ is constant and the quantities $\beta_{j,-}$, $\beta_{j,+}$ vary continuously except at moments when $y_0\in \{x_j\}_{j=1}^I\cap [-1,1]$.
Since $S$ is $\kappa$-subgaussian, if $\kappa$ is small enough with respect to $c>0$, then 
\begin{equation}\label{sumpi}
\sum_{x_i\leq -1} p_i<\frac{c}{10}\quad \text{and}\quad \sum_{x_i\geq 1} p_i < \frac{c}{10}.
\end{equation}
So when $y_0=-1$, by (\ref{b>c}),
$$\sum_{i=1}^{I} p_i \beta_{i,+} > c >\sum_{i=1}^{I_{y_0}}p_i\geq  \sum_{i=1}^{I_{y_0}}p_i(1-\alpha_i)$$
and when $y_0=1$, by (\ref{abb}) and (\ref{b>c}), 
$$\sum_{i=1}^{I} p_i \beta_{i,+} < \sum_{i=1}^{I} p_i (1-\alpha_i -c) \leq \sum_{i=1}^{I} p_i (1-\alpha_i) - c \leq \sum_{i=1}^{I_{y_0}}p_i(1-\alpha_i).$$
Thus by the intermediate value theorem, there ought to be some $y_0\in [-1,1]$ for which we have either Case (a) or Case (b).

In a way, Case (b) can be viewed as a degenerate case of Case (a) as follows. If Case (b) occurs, set $x'_j=x_j$ and $p'_j=p_j$ for all $j\in \{1,...,I_{y_0}-1\}$, and $x'_j=x_{j-1}$ and $p'_j=p_{j-1}$ for all $j\in \{I_{y_0}+2,...,I+1\}$.  Set $x'_{I_{y_0}} :=x_{I_{y_0}}$, $x'_{I_{y_0}+1} := x_{I_{y_0}}$ and $p'_{I_{y_0}} :=p'$, $p'_{I_{y_0}+1}=p_{I_{y_0}}-p'$.
Thus $x_{I_{y_0}}$ is repeated twice (by $x'_{I_{y_0}}$ and $x'_{I_{y_0}+1}$), and $p'_{I_{y_0}}$ and $p'_{I_{y_0}+1}$ are the probabilities corresponding to $x'_{I_{y_0}}$ and $x'_{I_{y_0}+1}$ respectively. Because of this, in the remaining of the proof, we will only need to explain the construction in Case (a).

Suppose then that Case (a) occurs.
Set 
\begin{align}\label{ggamma}
\begin{split}
\gamma_{-}& :=\int_{\R} \sum_{i=1}^I p_i g_{i,1}(x+x_i) \mathds{1}_{\{x<0\}}(x) \,dx = \sum_{i=1}^I p_i \beta_{i,-}>c,\\
\gamma_{+}&:=\int_{\R} \sum_{i=1}^I p_i g_{i,1}(x+x_i) \mathds{1}_{\{x\geq 0\}}(x)\, dx= \sum_{i=1}^I p_i \beta_{i,+}>c,
\end{split}
\end{align}
where the inequalities come from (\ref{b>c}).
Given $x_j$, let $V_j$ be a random variable with distribution density $\frac{1}{\alpha_j}g_{j,0}(x+x_j).$
Let $W_-$ and $W_+$ be random variables with distributions densities 
$\frac{1}{\gamma_{-}}\sum_{i=1}^{I}p_i g_{i,1}(x+x_i) \mathds{1}_{\{x<0\}}(x)$
and
$\frac{1}{\gamma_{+}}\sum_{i=1}^{I} p_i g_{i,1}(x+x_i) \mathds{1}_{\{x\geq0\}}(x)$ 
 respectively.
 Let $B$ be a random variable such that if our $\kappa$-subgaussian variable $S$ equals $x_j$, then $B = 1$ (resp. $0$) with probability $\alpha_j$ (resp. $1-\alpha_j$). 
Now, set $G$ to be a random variable such that 
\begin{enumerate}[label=(\roman*)]
\item conditioned on $S=x_j $ and $B=1$,  $G$ has same distribution as $V_j$,
\item conditioned on $S=x_j$, $j\geq I_{y_0}+1$ and $B=0$,  $G$ has  same distribution as $W_-$,
\item conditioned on $S=x_j$, $j \leq I_{y_0}$ and $B=0$,  $G$ has same distribution as $W_+$.
\end{enumerate}
By the previous conditions, the probability distribution of $G$ has density function 
\begin{align*}
& \sum_{j=1}^I p_j \alpha_j \frac{1}{\alpha_j}g_{j,0}(x+x_j) 
+\sum_{j=I_{y_0}+1}^{I} p_j (1-\alpha_j)\frac{1}{\gamma_{-}}\sum_{i=1}^{I}p_i g_{i,1}(x+x_i) \mathds{1}_{\{x<0\}}(x) \\
&+ \sum_{j=1}^{I_{y_0}} p_j (1-\alpha_j)\frac{1}{\gamma_{+}}\sum_{i=1}^{I}p_i g_{i,1}(x+x_i) \mathds{1}_{\{x\geq0\}}(x)\\
\underset{(\ref{casea})}{=} \, &\sum_{j=1}^I p_j g_{j,0}(x+x_j) 
+\sum_{i=1}^{I}p_i g_{i,1}(x+x_i) \mathds{1}_{\{x<0\}}(x) +\sum_{i=1}^{I}p_i g_{i,1}(x+x_i) \mathds{1}_{\{x\geq0\}}(x)\\
= &\sum_{j=1}^I p_j g_{j,0}(x+x_j) +\sum_{i=1}^{I}p_i g_{i,1}(x+x_i)\underset{(\ref{varphixxj})}{=} \, \varphi(x).
\end{align*}
Hence, $G$ is indeed a standard Gaussian random variable.
To finish the proof, let us check that $S+G$ has density $f(x)$ satisfying the bounds of the statement. We will use the standard notations $O(.)$ and $\Theta(.)$ as explained in Subsection \ref{sub:convention}.
The function $f$ is clearly well-defined and piecewise continuous. 
Assume first that $x< y_0$. We can write
$$f(x)=(i) + (ii) + (iii)$$
where 
\begin{align*}
\begin{split}
(i)  := \sum_{j=1}^{I} \frac{p_j \alpha_j}{\alpha_j}g_{j,0}(x) =\sum_{j=1}^{I} p_j g_{j,0}(x) 
\underset{x< y_0}{=} \nu \varphi(x-y_0) \sum_{-1\leq x_j<y_0} p_j,
\end{split}
\end{align*}
\begin{align*}
\begin{split}
(ii)& := \sum_{j=I_{y_0}+1}^I  \frac{p_j(1-\alpha_j)}{\gamma_{-}}\sum_{i=1}^I p_i g_{i,1}(x-x_j+x_i) \mathds{1}_{\{x-x_j<0\}}(x-x_j)\\
& \underset{x< y_0}{=} \sum_{j=I_{y_0}+1}^I \frac{p_j(1-\alpha_j)}{\gamma_{-}} \sum_{i=1}^I p_i g_{i,1}(x-x_j+x_i), \\
\end{split}
\end{align*}
\begin{align*}
\begin{split}
(iii)& :=\sum_{j=1}^{I_{y_0}} \frac{p_j(1-\alpha_j)}{\gamma_{+}} \frac{}{}\sum_{i=1}^I p_i g_{i,1}(x-x_j+x_i) \mathds{1}_{\{x-x_j\geq 0\}}(x-x_j) \\
&= \sum_{x_j\leq x} \frac{p_j(1-\alpha_j)}{\gamma_{+}} \sum_{i=1}^I p_i g_{i,1}(x-x_j+x_i). 
\end{split}
\end{align*}
To control $(i)$, note that 
\begin{align*}
1& \geq \sum_{-1\leq x_j<y_0} p_j 
=\sum_{j=1}^{I_{y_0}}p_j -\sum_{x_j<-1}p_j
\underset{(\ref{sumpi})}{\geq} \, \sum_{j=1}^{I_{y_0}}p_j -\frac{c}{10} \\
& \geq \sum_{i=1}^{I_{y_0}} p_j (1-\alpha_j)  -\frac{c}{10}
\underset{(\ref{casea})}{=} \, \sum_{i=1}^{I} p_j \beta_{i,+}   -\frac{c}{10}\underset{(\ref{b>c})}{\geq} \, \frac{9}{10}c>0.
\end{align*}
So for $x< y_0$,  (i) contributes $\Theta(\nu \varphi(x-y_0))$ to the density $f(x)$.
Next, to control $(ii)$, note that 
$$\sum_{i=1}^{I}p_i g_{i,1}(z+x_i)\leq \sum_{i=1}^{I}p_i \varphi(z) =\varphi(z) \quad \text{since $g_{i,1}(z+x_i) \leq \varphi(z)$,} $$
\begin{align*}
\begin{split}
0 & \leq \sum_{j=I_{y_0}+1}^I \frac{p_j(1-\alpha_j)}{\gamma_{-}} \varphi(x-x_j) 
\underset{(\ref{ggamma})}{\leq} \frac{1}{c}\sum_{j=I_{y_0}+1}^I p_j \varphi(x-x_j) \\
& \underset{x<y_0}{\leq} \frac{1}{c}\sum_{j=I_{y_0}+1}^I p_j \varphi(x-y_0)
\leq \frac{1}{c} \varphi(x-y_0).
\end{split}
\end{align*}
So for $x< y_0$, (ii) contributes $O(\varphi(x-y_0))$ to the density $f(x)$.
Next,  to control $(iii)$, we use that $S$ is $\kappa$-subgaussian\footnote{This is where $\kappa$-subgaussianity of $S$ is crucially used.} for some $\kappa\in (0,\frac{1}{2})$: we have $\sum_{x_j\leq x}p_j \leq 2 \exp(-\frac{x^2}{2\kappa})\leq  10^{100} \varphi(x-y_0)$ because $y_0\in [-1,1]$. Thus
$$0
\leq  \sum_{x_j\leq x} \frac{p_j(1-\alpha_j)}{\gamma_{+}}  \varphi(x-x_j) 
\underset{(\ref{ggamma})}{\leq}
\frac{1}{c}\sum_{x_j\leq x} p_j \varphi(x-x_j)
\leq  \frac{1}{c\sqrt{2\pi}}\sum_{x_j\leq x} p_j
\leq \frac{2.10^{100}}{c\sqrt{2\pi}}\varphi(x-y_0).$$
So for $x< y_0$, (iii) contributes $O(\varphi(x-y_0))$ to the density $f(x)$.
Putting these estimates together, the density $f(x)$ of the distribution of $S+G$ satisfies as desired:
$$\forall x< y_0,\quad C^{-1} \varphi(x-y_0)\leq f(x)\leq C \varphi(x-y_0)$$
for some universal constant $C\geq 1$.
With similar arguments, we conclude the same bound when $x> y_0$, which is enough to prove the lemma.

\end{proof}

Recall the notation $\varphi$ in (\ref{varphi}).
The following is due to Bobkov \cite[Lemma 3.1]{Bob10}\footnote{To get our statement from \cite[Lemma 3.1]{Bob10}, approximate the piecewise continuous density by continuous densities, set $c:=\log C^2$ and use that if $X\sim \mathcal{N}(0,e^{2c})$ then $\frac{1}{e^c}X\sim \mathcal{N}(0,1)$.}:
\begin{lem}[\cite{Bob10}]\label{thm:bobkov}
Let $C\geq 1$ and let $\mu$ be a probability measure on $\R$ with piecewise continuous density 
$$C^{-1}\varphi(x)\leq \frac{d\mu (x)}{dx} \leq  C \varphi(x).$$
Then $\mu = F_*(\gamma_1)$ for a Lipschitz map $F:\R\to \R$ with Lipschitz constant at most $C^2$.

\end{lem}

\begin{proof}[Proof of Theorem \ref{thm:1d}]
Let $\kappa>0$ and $C\geq 1$ be the universal constants in Lemma \ref{lem:density bound}. Consider a centered real-valued variable $S$ which is  $\frac{\kappa}{\sqrt{2}C^2}$-subgaussian, so that $\sqrt{2}C^2S$ is $\kappa$-subgaussian.
By Lemma \ref{lem:density bound} and using that if $G_0\sim \mathcal{N}(0,1)$ then $-G_0\sim \mathcal{N}(0,1)$, there is a standard Gaussian variable $G$ such that 
$\sqrt{2}C^2S-G$ has density $f:\R\to \R$ satisfying for some $y_0\in [-1,1]$:
 $$C^{-1}\varphi(x-y_0) \leq f(x) \leq C\varphi(x-y_0).$$
 By Lemma \ref{thm:bobkov}, there is a Lipschitz map $T:\R\to \R$ with Lipschitz constant at most $C^2$, sending the standard Gaussian measure $\gamma_1$ to the measure $f(x)dx$. By Lemma \ref{lem:average of gaussians} and since $S$ and $G$ are centered, this implies that there are two standard Gaussian variables $G',G''$ such that
 $\sqrt{2}C^2S-G = C^2\frac{G'+G''}{2}$, or equivalently $S = \frac{1}{\sqrt{2}C^2}(\frac{C^2}{2}(G'+G'')+G).$ Note that $\frac{C^2}{2\sqrt{2}C^2} = \frac{1}{2\sqrt{2}}\leq \frac{1}{\sqrt{2}}$ and $\frac{1}{\sqrt{2}C^2}\leq \frac{1}{\sqrt{2}}$. By Lemma \ref{lem:normalization},
$S\in \Sigma^3\mathcal{G}(\R)$ and the theorem is proved by renaming the universal constant $\frac{\kappa}{\sqrt{2}C^2}$ as $\kappa$.

\end{proof}

\subsection{Random vectors with bounded norm and covariance}
Recall that $\gamma_n$ is the standard Gaussian measure on $\R^n$.
A probability measure $\mu$ on $\R^n$ is called \emph{$1$-uniformly log-concave} if it has a density function that can be written as $\exp(-V(x)) $ where $V:\R^n \to \R\cup \{\infty\}$ is convex, $V$ is twice differentiable on the interior of  $\{V<\infty\}$ and $\Hess V\geq \Id$ on $\{V<\infty\}$.
We will need Caffarelli's contraction theorem \cite{Caf00}:
\begin{thm}[\cite{Caf00}]\label{caf}
Let $\mu$ be a probability measure on $\R^n$ which is $1$-uniformly log-concave. Then there is a $1$-Lipschitz map $F:\R^n\to \R^n$ such that $F_*(\gamma_n) = \mu$.
\end{thm}

By combining the Lipschitz image lemma (Lemma \ref{lem:average of gaussians}) and Caffarelli's contraction theorem (Theorem \ref{caf}), it is an exercise to show that for some real numbers $r_n$ bounded from below and above by universal positive constants, if $X$ is a centered random vector uniformly distributed on the vertices of a rescaled centered standard simplex in $\R^n$ such that $\|X\|=\sqrt{\log n}$ almost surely, then $r_nX$ is the sum of 3 standard Gaussian vectors in $\R^n$. 
That is a key observation in this section, which is generalized in the next lemma.
Recall the notation $\sum^q\mathcal{G}(\R^n)$ for $q$-fold Gaussian sums defined in (\ref{gaussian sum notation}).

\begin{lem} [$\sqrt{\log d}\,$-Bessel sequence]\label{bessel}

Given $C>0$, there is a positive integer $q=O(C)$ depending only on $C$ such that the following holds. 
Let $\Lambda\geq 1$, let $d,d_0$ be positive integers.
Given a linear map $F:\R^{d}\to \R^{d_0}$ with $\|F\| \leq C$,
define  $\hat{Z}_F$ to be the random vector uniformly distributed on the vectors 
$F(\sqrt{\log d} \, e_1)$, ..., $F(\sqrt{\log d}\,e_d)$.
Then $\hat{Z}_F-\E[\hat{Z}_F] \in \sum^q\mathcal{G}(\R^{d_0}).$
\end{lem}
\begin{proof}
Since if $k<m$, $\gamma_k$ on $\R^k\subset \R^m$ is the pushforward of $\gamma_m$ on $\R^m$ by an orthogonal projection map, we deduce thanks to Corollary \ref{cor:linear bis} that for this proof, we only need to deal with the case $d=d_0$. 
Thanks to Corollary \ref{cor:rescale}, we also only need to show the statement when $C\leq 1$. By Corollary \ref{cor:linear bis}, we can further reduce the proof  to the case where $F= \Id$.
Consider the vectors 
$$v_1:= \sqrt{\log d} \,(e_1-\frac{1}{d}\sum_{i=1}^d e_i),...,v_d:=  \sqrt{\log d} \,(e_d-\frac{1}{d}\sum_{i=1}^d e_i).$$ 
Let $Z_F$ be the random vector  uniformly distributed on the $v_i$'s.
To finish the proof, we just need to show that $Z_F\in \sum^q\mathcal{G}(\R^d)$ for some universal integer $q>0$.

Identify the subspace generated by $v_1,...,v_d$ with $\R^{d-1}$.  Given $j\in \{1,...,d\}$, consider the convex region
$$R_j:= \{z\in \R^{d-1};\quad \langle z,v_j\rangle\geq \langle z,v_k\rangle\text{ for all $k\neq j$}\}.$$
Let $G_1$ be a standard Gaussian random vector in $\R^{d-1}$. 
By symmetry, $G_1$ has equal probability $\frac{1}{d}$ of belonging to any of the $R_j$'s.
The underlying probability space $\Omega$ on which $G_1$ is defined can be divided into $d$ regions $\Omega_1,...,\Omega_d$ such that for $j\in \{1,...,d\}$, $\Omega_j$ is the region where $G_1\in R_j$. 
For $j\in \{1,...,d\}$, let $\mathbf{1}_{\Omega_j}$ denote the random variable equal to $1$ on $\Omega_j\subset \Omega$ and $0$ otherwise.
A standard computation and symmetry considerations show that for some $M>0$ independent of $d$, and some $C_d\in (\frac{1}{M},M)$,
\begin{equation}\label{averagex1}
\text{for $j\in \{1,...,d\}$, the expectation of $G_1$ conditioned on $\mathbf{1}_{\Omega_j}=1$ is equal to $C_d v_j$}.
\end{equation}
Given $j\in \{1,...,d\}$, let $\mu_j$ be the unique probability measure supported on $R_j$ and with density proportional to $\exp(-\|x\|^2/2)$ on $R_j$.  This measure $\mu_j$ is $1$-uniformly log-concave. By Theorem \ref{caf} and since $\gamma_{d-1}$ is the pushforward  of $\gamma_d$ by a $1$-Lipschitz map, there is a $1$-Lipschitz map $\Phi_j:\R^{d}\to \R^{d-1}$ such that 
\begin{equation}\label{phimuj}
(\Phi_j)_* \gamma_d = \mu_j.
\end{equation}
Clearly, the distribution of $G_1$, conditioned on $\mathbf{1}_{\Omega_j}=1$, is the probability measure $\mu_j$.
Hence, by Corollary \ref{cor:linear bis} and (\ref{averagex1}),
\begin{equation}\label{condj}
\text{$G_1-C_d v_j$ conditioned on $\mathbf{1}_{\Omega_j}=1$ belongs to ${\sum}^{2}\mathcal{G}(\R^d)$.} 
\end{equation}
To finish the proof, consider the random vector $Y_1$ in $\R^d$ defined on the probability space $\Omega$, equal to $G_1-C_d v_j$ on $\Omega_j$  for $j=1,..., d$.
Then by  (\ref{condj}) and Lemma \ref{lem:local global}, $Y_1\in {\sum}^{2}\mathcal{G}(\R^d)$.
Hence, $G_1-Y_1\in {\sum}^{3}\mathcal{G}(\R^d)$.
But  $G_1-Y_1$ has the same distribution as $C_d Z_F$, namely the uniform probability measure on $\{C_d v_1,...,C_d v_d\}$.  
The lemma is proved since $C_d\in (\frac{1}{M},M)$ for $M>0$ independent of $d$, and because of Corollary \ref{cor:rescale}.

\end{proof}

We will use a corollary of the work of Marcus-Spielman-Srivastava \cite{Mar15}:
\begin{thm}[\cite{Mar15}]\label{thm:mss}
Let $k$,  $m$, $n$ be positive integers such that $m\geq k$, and let $v_1,...,v_m\in \R^n$ be such that   
$$\|v_i\|^2\leq 1\quad \text{for all $i=1,...,m,$ and}\quad \frac{1}{m}\sum_{i=1}^m v_iv_i^T \leq \frac{1}{k} \Id.$$
Then there exists a partition $\{T_1,...,T_s\}$ of $[m]$ for some integer $s>0$, such that for some universal constant $\tilde{C}\geq 1$, 
$$\tilde{C}^{-1}k< |T_j| <\tilde{C} k \quad \text{and}\quad   \|\frac{1}{|T_j|}\sum_{i\in T_j} v_iv_i^T\| \leq  \frac{\tilde{C}}{k} \quad \text{for each $j=1,...,s$}.$$
\end{thm}
\begin{proof}
Let $e_1,...,e_k$ be the standard basis of $\R^k$.
Consider the $m$ vectors $u_1,...,u_m\in \R^k$ defined as $u_i = e_j$ if $i$ equals $j\in \{1,...,k\}$ modulo $k$. One checks that
\begin{equation}\label{nondeg}
\frac{1}{2k} \Id \leq \frac{1}{m}\sum_{i=1}^m u_iu_i^T \leq  \frac{2}{k} \Id.
\end{equation}
Set $w_i:=(u_i,v_i)\in \R^{k+n}$. 
Since $\|v_i\|^2\leq 1$ and $\|u_i\|^2 = 1$, we have $1 \leq \|w_i\|^2\leq 2$. Moreover, using Cauchy-Schwarz, one readily checks from (\ref{nondeg}) that
$\frac{1}{m}\sum_{i=1}^m w_iw_i^T \leq  \frac{6}{k} \Id.$
Next, we  apply \cite[Corollary 1.5]{Mar15} (which holds more generally for sub-isotropic families of vectors, as can be seen by adding dummy vectors)
to the $w_i$'s and to $r:= \lfloor \frac{m}{k} \rfloor$: we get a partition $\{S_1,...,S_r\}$ of $[m]$ such that for each $j=1,...,r$,
\begin{equation}\label{bd sj}
 \|\sum_{i\in S_j} w_iw_i^T\| \leq  50.
 \end{equation}
This upper bound implies $\sum_{i\in S_j} u_iu_i^T \leq 50\Id$ which gives after taking the trace in $\R^k$:
\begin{equation}\label{10k}
\sum_{i\in S_j} \|u_i\|^2 = |S_j| \leq 50k.
\end{equation}
 On the other hand, $\sum_{j=1}^r |S_j| =m\geq rk$ so by (\ref{10k}), after renumbering the $S_i$'s so that $|S_j|$ is nonincreasing in $i$, we have for all $t=1,...,\lfloor \frac{r}{100}\rfloor+1$,  
\begin{equation}\label{si large}
|S_t|\geq \frac{k}{3}.
\end{equation}
We can arbitrarily partition $\{1,...,r\}\setminus \{1,...,\lfloor \frac{r}{100}\rfloor+1\}$  into subsets $O_1,...,O_{\lfloor \frac{r}{100}\rfloor+1}$ of size at most $100$, so that by (\ref{10k}), for $t=1,...,\lfloor \frac{r}{100}\rfloor+1$, 
\begin{equation}\label{1000k}
|\bigcup_{j\in O_t} S_j| \leq 5000k.
\end{equation}
For $t=1,...,\lfloor \frac{r}{100}\rfloor+1$, define
$T_t:= S_t\cup \bigcup_{j\in O_t} S_j.$
Then by (\ref{10k}), (\ref{si large}) and (\ref{1000k}), we have $\frac{k}{3}< |T_t| <5050k$. 
Finally this bound, the fact that $|O_t|\leq 100$, and (\ref{bd sj})  imply that $\|\frac{1}{|T_t|}\sum_{i\in T_t} v_iv_i^T\| \leq  \frac{3\times 5050}{k} $ as desired.

\end{proof}

\begin{proof}[Proof of Theorem \ref{thm:general}]

By an approximation and compactness argument, we can assume that $X$ has finite support $\supp X\subset \R^n$ and that $X$ is the uniform probability measure on $\supp X$.  
Fix $\epsilon>0$ and suppose that $\supp X= \{v'_1,...,v'_m\}\subset \R^n$ and  $m:=|\supp X|$. By adding copies of the $v'_i$'s, we can assume that $m\geq \lfloor e^{\Lambda^2}\rfloor$.
By our assumptions on $X$,
$$\|v'_i\|\leq \Lambda\quad \text{for all $i=1,...,m,$ and}\quad  \frac{1}{m} \sum_{i=1}^m (v'_i)(v'_i)^T \leq \Lambda^2e^{-\Lambda^2} \Id.$$
By Theorem \ref{thm:mss} applied to $v_i:=v'_i/\Lambda$ and  $k:= \lfloor e^{\Lambda^2}\rfloor$,
there exists a partition $\{T_1,...,T_s\}$ of $[m]$ for some integer $s>0$, such that for some universal $\tilde{C}\geq 1$, for each $j=1,...,s$,
$$\tilde{C}^{-1} e^{\Lambda^2}<|T_j| <\tilde{C} e^{\Lambda^2} \quad \text{and}\quad   \|\sum_{i\in T_j} (v'_i)(v'_i)^T\| \leq  \tilde{C} \Lambda^2 e^{-\Lambda^2} |T_j|\leq \tilde{C}^2 (\log |T_j|+\log \tilde{C})$$
where the last inequality follows from the upper and lower bounds on $|T_j|$.
But by basic linear algebra \cite[Definition 1.15, Lemma 1.3]{Cas13},
the upper bound for $\|\sum_{i\in T_j} (v'_i)(v'_i)^T\|$ means that for some universal integer $C_0>0$,  for each $j=1,...,s$, setting $d_j:=|T_j|$ and fixing an arbitrary bijection $\sigma_j:T_j\to [d_j]$, there is a linear map $F_j:\R^{d_j}\to \R^n$ with $\|F_j\|\leq C_0$ such that 
$$\forall i\in T_j,\quad v'_i = F_j(\sqrt{\log d_j}\, e_{\sigma_j(i)}) \quad \text{where $(e_1,...,e_{d_j})$ is the standard basis of $\R^{d_j}$}.$$
Let $Z_j$ be the random vector uniformly distributed on $\{v'_i\}_{i\in T_j}$, defined as independent random vectors on a common probability space.
Note that 
$$X= \sum_{j=1}^s \mathbf{1}_{\{\mu = j\}} Z_j$$ where $\mu$ is a random variable independent from the $Z_j$'s, and $\mu$ is equal to $j$ with probability $|T_j|/{m}$. An easy computation shows that $\|\E[Z_j]\|\leq C_0 \quad \text{for all $j\in \{1,...,s\}$}$.
By Lemma \ref{bessel} applied to the dimension $d_j$ and the linear map $F_j$, for all $j\in \{1,...,s\}$,  $Z_j - \E[Z_j] \in \sum^{q_0}\mathcal{G}(\R^n)$ for some universal integer $q_0>0$. 
By Corollary \ref{cor:bdd support II}, $X\in \sum^{q_0+10C_0}\mathcal{G}(\R^n)$
 and the theorem is proved.

\end{proof}

\begin{rem} [Subgaussian = finite sum of Gaussians] \label{rem:charac}
Using that any $1$-subgaussian random vector $X$ in $\R^n$  has rapidly decaying tail probabilities, after rescaling $X$ by a constant depending on $n$, we obtain a random vector which is equal to a mixture of random vectors satisfying the assumptions of Theorem \ref{thm:general}. Thus Theorem \ref{thm:general} already implies that any $1$-subgaussian random vector is the sum of finitely many standard Gaussian vectors. The number of standard Gaussian vectors given by 
this argument is $O(\sqrt{ n /\log n})$. 
\textcolor{RedViolet}{\textbf{Added May 2026:} Of course, the follow-up paper \cite{HST26} now gives the optimal bound $O(1)$.}
\end{rem}

\subsection{Further discussions}

\subsubsection{Geometric interpretation}\label{subsection:geometric}
Problem \ref{S} aims to provide fundamental geometric information about the space of Gaussian random vectors in $\R^n$. 
Let $\Omega$ be $[0,1]$ endowed with the Lebesgue measure.
Let $\mathcal{G}(\R^n)\subset L^2(\Omega,\R^n)$ be the set of all $n$-dimensional standard Gaussian vectors in $\R^n$. Let $\overline{\conv}(\mathcal{G}(\R^n))$ be the closure of its convex hull, whose points are $1$-subgaussian random vectors \cite{Bul00}, and which conversely contains the set of $1$-subgaussian vectors in $\R^n$ up to scaling by universal constant, see Subsection \ref{rem:un} below. 
Theorem \ref{thm:1d} implies that for some universal integer $L>0$,
$$\frac{1}{L}\overline{\conv}( \mathcal{G}(\R)) \subset  \mathcal{G}(\R)+\mathcal{G}(\R)+\mathcal{G}(\R).$$
Problem \ref{S} asks whether there is a universal integer $q>0$ such that in any dimension $n\geq 2$,
$$\overline{\conv}( \mathcal{G}(\R^n))  \subset {\underbrace{\vphantom{\Big|}\mathcal{G}(\R^n)+\cdots+\mathcal{G}(\R^n)}_{q\ \text{times}}}.$$

\subsubsection{The convex hull of standard Gaussian vectors} \label{rem:un}
Averages (i.e. finite convex linear combinations) of standard Gaussian random vectors are well-known to be  $\kappa'$-subgaussian for some universal $\kappa'>0$, see \cite[Theorem 1.2]{Bul00} \cite[Exercise 2.42]{Ver18}. In this subsection, we explain that conversely,
for a universal $\kappa>0$, any $\kappa$-subgaussian random vector $X$ in $\R^n$ can be approximated by an average  of standard Gaussian random vectors. In other words, $\kappa$-subgaussian random vectors essentially form the ``convex hull'' of the space of standard Gaussian vectors.  
This important fact readily follows from (and is in fact equivalent to) the following strengthening of the subgaussian comparison theorem of Talagrand observed by van Handel \cite[Corollary 1.2]{Van25}: for a universal $\kappa>0$, any $\kappa$-subgaussian random vector $X$ in $\R^n$ is dominated in the convex order by a standard Gaussian vector.
Alternatively, here is a quick sketch of a proof  using \cite{Tal21} and  a tensorization trick in the spirit of the proof of Theorem \ref{thm:equivalence}.
Fix $\delta>0$. If $X$ is $1$-subgaussian in $\R^n$, and $M$ is a large integer, consider the set $\mathcal{S}$ of points $x=(x_1,...,x_M)$ in $\R^{nM}$ such that the empirical distribution of the $x_i$ is $1$-Wasserstein $\delta$-close to the distribution of $X$. Consider i.i.d. copies of $X$, called $X_1,...,X_M$. 
Then with high probability by Glivenko-Cantelli's theorem, $\hat{X}:=(X_1,...,X_M)$ is in $\mathcal{S}$ when $M$ is large. Next, consider the convex hull $\conv(\mathcal{G})$ of the set $\mathcal{G}$ of vectors $x=(x_1,...,x_M)$ in $\R^{nM}$ such that the empirical distribution of the $x_i$ is $1$-Wasserstein $\delta$-close to the distribution of the standard Gaussian vector  in $\R^n$. Then by Glivenko-Cantelli again, the Gaussian volume of $\conv(\mathcal{G})$ is at least $1/2$ for $M$ large. 
Since $\hat{X}$ is $1$-subgaussian in $\R^{nM}$, by Theorem \ref{tal}, $\kappa \,\hat{X}$ belongs to $\conv(\mathcal{G})$ with uniformly positive probability ($\kappa>0$ is universal). This means that $\kappa\,\hat{X}$ is a convex combination of some points in $\mathcal{G}$, which implies the desired result after taking $\delta\to 0$.

\subsubsection{Simple random vectors}\label{rem:generalize mss?}

Lemma \ref{bessel} generalizes as follows. Consider a partition of $\R^d$ into convex regions $\{R_k\}_{k\geq 1}$.
Let $G\sim \mathcal{N}(0,I_d)$, let $p_k:=\mathbb{P}[G \in R_k]$, let  $x_k:=\E[\, G\, | G \in R_k] \in \R^d$, let $S$ be the random vector in $\R^d$ equal to $x_k$ with probability $p_k$, 
Then $S\in \sum^3\mathcal{G}(\R^n)$.
Moreover, if $F:\R^d \to \R^n$ is any linear map with $\|F\|\leq 1$ and  
if $\hat{Z}$ is the random vector equal to $F(S)$, 
then $\hat{Z}\in \sum^q\mathcal{G}(\R^n)$ for some universal $q>0$. Let us call any such $\hat{Z}$ a \emph{simple random vector}.
Now, the Marcus-Spielman-Srivastava theorem (Theorem \ref{thm:mss}) implies that, up to a uniform scaling, any random vector $X$ in $\R^d$ satisfying\footnote{The symmetry assumption $X\sim -X$ is for the sake of simplicity.}  $X\sim -X$, $\|X\|\leq \Lambda\in [1,\infty)$ almost surely and $\|\Cov X\|\leq \Lambda^2 e^{-\Lambda^2}$, can be decomposed into a mixture of simple random vectors (see proof of Theorem \ref{thm:general}).
To settle Problem \ref{S}, it would suffice to establish a subgaussian analogue: for $\kappa>0$ small enough, for any $d\geq1$ and any $\kappa$-subgaussian random vector $X$ in $\R^d$ with $X\sim -X$, is $X$ a mixture of simple random vectors? Is there at least one simple random vector whose support is contained in the support of $X$?

\textcolor{RedViolet}{\textbf{Added May 2026:} This strategy does in fact lead to a full solution to Problem \ref{S}! This is outlined in \cite[Appendix B]{HST26}. 
Simple random vectors defined as above are directly connected to ``Laguerre tessellations'', which are well-studied in optimal transport and materials science, see e.g. \cite{BGMN25}. What can be shown is that for $\kappa>0$ small enough, any $\kappa$-subgaussian random vector in $\R^n$ can be approximated by simple random vectors which admit a very simple description in terms of orthogonal projections of Laguerre tessellations \cite[Appendix B]{HST26}.}

\section{How large are sums of large sets?}

We explain in this section some  geometric consequences  of Section \ref{section:expressive}. The general principle is that the expressivity of sums of Gaussian vectors accounts for the size of sums of large sets in Gaussian spaces.

\subsection{Large permutation invariant sets}\label{subsection:permutation}
In \cite[Proposition 2.10]{Tal95}, Talagrand outlined an explanation of how Theorem \ref{thm:1d} implies Corollary \ref{cor:permutation}. However, the proof of \cite[Proposition 2.10]{Tal95} crucially relies on a uniform $O(1)$ bound \cite[Lemma 2.11]{Tal95}, which is not correct\footnote{I would like to thank Samuel Johnston for showing me the reference \cite{BF20} and an unpublished note, which is how I realized the issue with \cite[Lemma 2.11]{Tal95}.}. The right bound $O(\log \log n)$ is given in Lemma \ref{corrrect} below.

Set $\Phi(x):=\gamma_1(-\infty,x)$. Given an integer $n\geq 2$ and real-valued random variables $X_1,...,X_n$, denote by $X^*_1,...,X^*_n$ the random variables given by reordering the $X_i$'s in non-decreasing order.
\begin{lem}\label{corrrect}
Let $n\geq 2$ be a positive integer. Consider probability measures $\theta_i$ on $\R$ ($i=1,...,n$), such that $\sum_{i=1}^n \theta_i = n\gamma_1$. Define $r_i\in \R$ such that $\Phi(r_i)=\frac{i}{n}$ if $i\leq \frac{n}{2}$, and $\Phi(r_i)=\frac{i-1}{n}$ if $i> \frac{n}{2}$. 
Consider independent random variables $X_i$ with respective probability distributions $\theta_i$. Then for some universal constant $C_c\geq 1$,
$$\sum_{i=1}^n \E[|X^*_i - r_i|^2]\leq C_c (\log \log n +1).$$  
\end{lem}
This lemma amends  \cite[Lemma 2.11]{Tal95} and partially generalizes  \cite{DBGU05}\cite[Corollary 6.14]{BL19}\cite[Theorem 1]{BF20} which focus on the cases where $\theta_i =\gamma_1$ for all $i$. As shown in these references, the $O(\log \log n)$ bound above is optimal already when $\theta_i =\gamma_1$ for all $i$. The proof is postponed to the Appendix.

Next, we will need the following corollary of Theorem \ref{thm:general}:
\begin{cor}\label{lem:salvage}
Let $n\geq 1$ and let $A\subset \R^n$ be a permutation invariant set such that $\gamma_n(A)\geq 2/3$. 
Then there exists a universal integer $\hat{q}>0$ such that
$$ B(0,\sqrt{\log n}) \cap \{x_1+...+x_n=0\} \subset A_{(\hat{q})}.$$
\end{cor}
\begin{proof}
Pick a nonzero vector $(y_1,...,y_n)\in B(0,\sqrt{\log n}) \cap \{x_1+...+x_n=0\}$.
Consider the random vector $Y$ with probability distribution $\frac{1}{n!}\sum_{\sigma} \delta_{(y_{\sigma(1)},...,y_{\sigma(n)})}$ where the sum is over all permutations of $[n]$. Then by construction, $\|Y\|\leq \sqrt{\log n}$ almost surely. By permutation invariance, $\Cov Y$ has two eigenvalues: $0$ which corresponds to the diagonal direction $(1,...,1)$ and $\lambda>0$ whose corresponding eigenspace is $\{x_1+...+x_n=0\} $. Since $\Tr \Cov Y = \E[\|Y\|^2]\leq \log n$ by the norm bound, we deduce that $\lambda = O(\frac{\log n}{n})$. In other words $\|\Cov Y\| = O(\frac{\log n}{n})$. We can then apply Theorem \ref{thm:general} and get $Y= G_1+...+G_{q'}$ for some universal $q'>0$ and $G_i\sim \mathcal{N}(0,I_n)$. On the other hand, Lemmas \ref{Steinhaus} and \ref{cor:vol sum} imply that $\gamma_n(A_{(q'')})> 1-\frac{1}{q'}$ for some universal $q''>0$.
Thus the union bound ensures the existence of points $g_1,...,g_{q'}\in A_{(q'')}$ such that $(y_{\sigma(1)},...,y_{\sigma(n)}) = g_1+...+g_{q'}$
for some permutation $\sigma$ of $[n]$, which means that $(y_{\sigma(1)},...,y_{\sigma(n)}) \in A_{(q'q'')}$. By permutation invariance of $A$, $y\in A_{(q'q'')}$ and the lemma follows after taking $\hat{q}:=q'q''$.

\end{proof}

\begin{proof}[Proof of Corollary \ref{cor:permutation}]
We will borrow the original argument of \cite[Proposition 2.10]{Tal95} but we will add a new ingredient, Corollary \ref{lem:salvage}, to repair the end of the argument. 
Let $n\geq 2$ and let $A$ be a permutation invariant set with $\gamma_n(A)\geq 2/3$.  
Let $\Omega\subset \R^n$ be the set of all points $y=(y_1,...,y_n)$ such that $y\in \{x_1+...+x_n=0\}$ and $\frac{1}{n}\sum e^{y_i^2/10}\leq 2$ (in particular, the random variable $Y$ with probability distribution
$\frac{1}{n}\sum \delta_{y_i}$ is centered and $\alpha$-subgaussian for some universal $\alpha>0$, see \cite[Proposition 2.6.1]{Ver18}). 
For example, if $r_i\in \R$ are defined as in Lemma \ref{corrrect}, then $(r_1,...,r_n)\in \Omega$  by the Mills ratio inequalities (\ref{mills}). 
Denote as usual the Euclidean $R$-ball centered at $0$ by $B(0,R)$.  
Let $C_c\geq 1$ be the universal constant given in Lemma \ref{corrrect}.
Set 
$$\hat{B}:=B(0, 100\sqrt{C_c}\sqrt{\log \log n+1})\cap \{x_1+..+x_n=0\},$$
\begin{align}\label{defKK}
K:= \Omega+\hat{B}+B(0,100).
\end{align}
This set $K$ is a convex body since $\Omega$ is convex (see \cite[Theorem 1.2]{Bul00} \cite[Exercise 2.42]{Ver18}). By the central limit theorem and by Lemma \ref{corrrect} applied to the case  $\theta_i=\gamma_1$ (or by \cite[Corollary 6.14]{BL19} \cite[Theorem 1]{BF20}), we have $\gamma_n(K)\geq 1/2$. 

Next, it remains to show that $K\subset A_{(q)}$ for some universal integer $q>0$. Since $\log\log n = o(\log n)$, Corollary \ref{lem:salvage} (combined with Lemma \ref{Steinhaus}) yields $\hat{B}+B(0,100) \subset A_{(q_0)}$ for some universal $q_0>0$. So to finish the proof, it would be enough to prove that 
\begin{align}\label{333}
\Omega\subset \big(A+\hat{B}+B(0,100)\big)_{(q)}
\end{align} 
for some universal $q>0$. 
Fix a point $y=(y_1,...,y_n)\in \Omega\subset \{x_1+...+x_n=0\}$. Let us assume for simplicity that all the $y_i$ are different (the general case will follow from a compactness argument).
By\footnote{To prove Corollary \ref{cor:permutation}, we could also use Remark \ref{rem:charac} in dimension $1$,  instead of Theorem \ref{thm:1d}.} Theorem \ref{thm:1d}, there is a universal $\kappa'\in (0,1)$ such that the random vector $X$ with probability distribution
$\nu:= \frac{1}{n}\sum \delta_{\kappa'y_i}$ is the sum of three standard Gaussian random variables. This means that there is a  probability measure $\mu$ on $\R^3$ such that each of the three $1$-dimensional marginals of $\mu$ is $\gamma_1$, and $\nu$ is the pushforward of $\mu$ by the map $\mathcal{S}: (s,t,u)\mapsto s+t+u$. For each $i=1,...,n$, let $\mu_i$ be the probability measure given by $\mu_i(U) := n\mu(U\cap \mathcal{S}^{-1}(\kappa'y_i))$ for any Borel set $U\subset \R^{3}$. Consider the product probability measure $\bar{\mu}:= \otimes_{i=1}^n \mu_i$ on $\R^{3n}=(\R^n)^{3}$.
By construction, for $\bar{\mu}$-almost every point $(x_1,...,x_{3n})\in \R^{3n}$, we have $\kappa'y= (x_1,...,x_n)+(x_{n+1},...,x_{2n})+(x_{2n+1},...,x_{3n})$. Thus, in view of our goal (\ref{333}), it would suffice to show:
\begin{align}\label{444}
\begin{split}
& \text{with positive $\bar{\mu}$-probability, a random $x=(x_1,...,x_{3n})\in\R^{3n}$ satisfies}\\
&\quad 
(x_{(j-1)n+1},...,x_{(j-1)n+n})\in A+\hat{B}+B(0,100) \quad \text{for any $j=1,2,3$}.
\end{split}
\end{align}  
To do so, let $\theta_{i,1}, \theta_{i,2}, \theta_{i,3}$ be the $1$-dimensional marginals of $\mu_i$. By construction, for each $j=1,2,3$,
\begin{align}\label{sumgamma1}
\sum_{i=1}^n \theta_{i,j} = n\gamma_1.
\end{align} 
For $j=1,2,3$, consider the product probability measure $\bar{\theta_j}:= \otimes_{i=1}^n \theta_{i,j}$ on $\R^{n}$, namely the marginals  of $\bar{\mu}$ on each factor $\R^n$.
To show (\ref{444}), it is enough to check that for each $j=1,2,3$:
\begin{align}\label{555}
\begin{split}
&\text{with $\bar{\theta_j}$-probability larger than $2/3$, a random $z\in\R^n$ satisfies }\\ & \quad  z \in A+\hat{B}+B(0,100).
\end{split}
\end{align} 
For $j=1,2,3$, for a random point $z=(z_1,...,z_n)\in\R^n$ with probability distribution $\bar{\theta_j}$, 
\begin{align}\label{totvar}
\E_{\bar{\theta_j}}(\frac{1}{n}(z_1+...+z_n)^2) = \frac{1}{n} \sum \Var_{\bar{\theta_j}}(z_i)\leq 1
\end{align} 
where the  equality uses that $z_i\sim \theta_{i,j}$ are independent, and the inequality uses (\ref{sumgamma1}) and the law of total variance.
Let $r_i$ be defined as in Lemma \ref{corrrect}. 
Thanks to (\ref{sumgamma1}), by applying Lemma \ref{corrrect} to the case $\theta_i=\theta_{i,j}$, and by (\ref{totvar}), for a random $z\in \R^n$, with  $\bar{\theta_j}$-probability larger than $2/3$,
 the Euclidean distance between $z$ and $\{x_1+...+x_n=0\}$ is at most $50$,
 and 
$\|z -(r_{\sigma(1)},...,r_{\sigma(n)})\|\leq 50 \sqrt{C_c} \sqrt{\log \log n +1}$ for some permutation $\sigma$ of $[n]$.
Similarly since $\gamma_n(A)\geq 2/3$, by the central limit theorem, by Lemma \ref{corrrect} applied to the case $\theta_i=\gamma_1$ (or by \cite{BL19}\cite{BF20}), and by permutation invariance of $A$,
 there is some point $a\in A$ whose Euclidean distance to $\{x_1+...+x_n=0\}$ is at most $50$,
and 
$\|a -(r_{\sigma(1)},...,r_{\sigma(n)})\|\leq 50 \sqrt{C_c} \sqrt{\log \log n +1}$. 
Combining these estimates together, we get (\ref{555}). The proof is finished.

\end{proof}

From (\ref{defKK}), we see that the convex body $K$ obtained in Corollary \ref{cor:permutation} can be chosen \emph{explicitly}, as a set of points whose coordinates approximate a $\kappa$-subgaussian variable. In contrast, the ellipsoid of Corollary \ref{cor:large ellipsoid} will not be constructed explicitly. 

\subsection{Ellipsoids in sums of large sets}

Let us move on to Corollary \ref{cor:large ellipsoid}. In preparation for the proof, we will need the next elementary lemma, which is the ellipsoid version of the more elaborate \cite[Theorem 1.2]{Dad19}. In this paper, an \emph{ellipsoid} $E$ is by definition any set of the form 
$$E:=\{x^TQx \leq 1\}\subset \R^n$$
where $Q$ is some symmetric positive semidefinite $n$-by-$n$ matrix. We define the trace $\Tr E$ of the ellipsoid $E$  as the trace $\Tr Q$. 
\begin{lem}[Ellipsoid and Covariance]\label{cor:ellcov}
Given a closed symmetric set  $S\subset \R^n$ and  $\tau>0$, the following are equivalent:
\begin{enumerate}
\item Any ellipsoid $E$ with trace $\Tr E=\tau$ intersects $S$.
\item There is a random vector $X$ supported on $S$ such that $$X\sim -X\quad\text{and}\quad \|\Cov X\|\leq \tau^{-1} .$$
\end{enumerate}
\end{lem}
\begin{proof}
By a standard compactness argument, it is enough to treat the case where $S$ is finite. The lemma then follows from the von Neumann minimax theorem, in a similar fashion as \cite[Theorem 1.2]{Dad19}. 
Let $\mathcal{Q}_\tau$ be the set of symmetric positive semidefinite $n$-by-$n$ matrices with trace equal to $\tau$, and let $\Delta_S$ be the simplex of symmetric probability measures on $S$. These are compact convex sets and any $\mu\in \Delta_S$ has average $0$. The minimax principle applied to the affine function 
$$f:\mathcal{Q}_\tau\times \Delta_S$$
$$f(Q,\mu) := \Tr(Q\Cov(\mu))$$
gives 
$$\max_{Q\in \mathcal{Q}_\tau} \min_{\mu \in \Delta_S} \Tr(Q\Cov(\mu)) = \min_{\mu \in \Delta_S} \max_{Q\in \mathcal{Q}_\tau} \Tr(Q\Cov(\mu)).$$
But $\max_{Q\in \mathcal{Q}_\tau} \Tr(Q\Cov(\mu)) = \tau \|\Cov(\mu)\|$, and 
$\min_{\mu \in \Delta_S} \Tr(Q\Cov(\mu)) = \min_{s \in S} \Tr(Q ss^T) =  \min_{s \in S}  s^TQ s$, so 
$$\max_{Q\in \mathcal{Q}_\tau} \min_{s \in S}  s^TQ s= \min_{\mu \in \Delta_S} \tau \|\Cov(\mu)\|.$$
From this equality, the equivalence is clear.

\end{proof}

Recall the notation $A_{(q)}$ for the $q$-fold Minkowski sum of $A$, defined in (\ref{sum notation}).
\begin{proof}[Proof of Corollary \ref{cor:large ellipsoid}]
An ellipsoid $E=\{x^TQx \leq 1\}$ with trace $\Tr E = \Tr Q = \frac{1}{10}$ has Gaussian measure $\gamma_n(E)=\mathbb{P}[\sum_{i=1}^n \lambda_i Z_i^2\leq 1] = \mathbb{P}[\sum_{i=1}^n 10\lambda_i Z_i^2\leq 10] $ where the $\lambda_i$'s are the eigenvalues of $Q$ and the $Z_i$'s are independent standard Gaussian variables. Since $\sum_{i=1}^n 10\lambda_i =1$, \cite[Corollary 3]{Sze03} implies that 
$$\gamma_n(E)\geq \mathbb{P}[Z_1^2\leq 10]\geq \frac{1}{2}$$
where the second inequality follows from well-known estimates for the Gaussian measure.
Moreover by Lemmas \ref{Steinhaus} and \ref{cor:vol sum}, if a set $A$ satisfies $\gamma_n(A)\geq \frac{2}{3}$ then $\gamma_n(A') \geq \frac{2}{3}$ where $A':=A_{(k_0)}\cap (-A_{(k_0)})$ and $k_0>0$ is a universal integer. The point here is that $A'$  is symmetric.

In view of these preliminary remarks, in order to prove the corollary, it is sufficient to show that for some universal integer $q_0>0$, for a symmetric set $A\subset \R^n$ with $\gamma_n(A)\geq \frac{2}{3}$, 
there is an ellipsoid $E$ in $\R^n$ with  
\begin{equation}\label{tre110}
\Tr E = \frac{1}{10} \quad \text{and}\quad   \sqrt{\frac{\log n}{n}} E\subset  A_{(q_0)}.
\end{equation}
Let $q>0$ be the integer given in Theorem \ref{thm:general}. 
By Lemmas \ref{Steinhaus} and \ref{cor:vol sum}, for some universal $k_1\geq 2$, 
\begin{equation}\label{k_1bis}
\gamma_n(A_{(k_1)}) >1-\frac{1}{10q}.
\end{equation}
Suppose towards a contradiction that (\ref{tre110}) fails for $k_0 = 10k_1q+2$. 
Then for any ellipsoid $E$ with $\Tr E=\frac{1}{10}$, $E\cap \R^n\setminus \sqrt{\frac{n}{\log n}}A_{(10k_1q+2)} \neq \varnothing$. By Lemma \ref{cor:ellcov}, there is a random vector $X_0$ in $\R^n$ whose support $\supp X_0$ is contained in the closure of $\R^n\setminus  \sqrt{\frac{n}{\log n}} A_{(10k_1q+2)}$,
such that $X_0\sim -X_0$ and $\|\Cov(X_0) \|\leq 10$. 
Note that by Lemma \ref{Steinhaus}, $A_{(10k_1q)}$ is strictly contained in the interior of $A_{(10k_1q+2)}$, meaning that 
\begin{equation}\label{xa(q)}
\sqrt{\frac{\log n}{n}}  \supp X_0 \subset \R^n\setminus  A_{(10k_1q)}.
\end{equation} 
Set 
$$X_1:=  \mathbf{1}_{\{\|X_0\|^2 \leq 100n\}} X_0$$
which satisfies $X_1\sim -X_1$ like $X_0$. In particular $X_1$ is centered.
By the covariance bound for $X_0$ and Markov's inequality, $X_1=X_0$ with probability at least $\frac{1}{2}$. Besides, $\|\frac{1}{10}\sqrt{\frac{\log n}{n}} X_1\|\leq \sqrt{\log n}$ almost surely, and $\|\Cov(\frac{1}{10}\sqrt{\frac{\log n}{n}} X_1)\|\leq \frac{\log n}{n}$, so applying Theorem \ref{thm:general} to the centered  random vector $\frac{1}{10}\sqrt{\frac{\log n}{n}} X_1$, 
\begin{equation}\label{lognnq}
\sqrt{\frac{\log n}{n}}X_1=10(G_1+...+G_q)
\end{equation} 
for some standard Gaussian random vectors $G_i$. 
By (\ref{k_1bis}) and the union bound, we deduce that $G_i\in A_{(k_1)}$ for all $i\in\{1,...,q\}$ and $X_1=X_0$ at the same time with  positive probability,
 so in particular by (\ref{lognnq}), there are points $a_1,...,a_q\in A_{(k_1)}$ and $x\in \supp X_0$,  with 
$$\sqrt{\frac{\log n}{n}}x = 10(a_1+...+a_q) \in 10A_{(k_1q)}\subset A_{(10 k_1q)}.$$
But this contradicts (\ref{xa(q)}).
This finishes the proof.

\end{proof}

\subsection{Further discussions}

\subsubsection{Largest slices} \label{rem:deux}
One can show that Corollary \ref{cor:large ellipsoid} and the standard  Lemma \ref{cor:vol sum} imply an optimal bound for the largest sections of sums of large sets: there exists an integer $q>0$ such that 
for any $\theta\in (0,1)$,
 if $A$ is a closed set in $\R^n$ with $\gamma_n(A)\geq \frac{2}{3}$, then there is a linear  subspace $H$ with $\dim H = \lfloor \theta n \rfloor$ and  $c=c(\theta)>0$ such that
\begin{equation}\label{bdomega}
 \forall k\geq q,\quad \gamma_{H}\big(A_{(k)}\cap H\big)\geq 1-n^{-c k^2}.
 \end{equation}
Up to the constant $c>0$, this bound  is optimal:  consider  the Gaussian measure of sections of centered hypercubes \cite[Theorem 10]{Bar05}. 
This remark also provides an indirect way to verify that the factor $\sqrt{{n}^{-1}{\log n}}$ in Corollary \ref{cor:large ellipsoid} is sharp up to uniform constants already for centered hypercubes
(if $\sqrt{{n}^{-1}{\log n}}$  could be replaced  by a much larger constant, then it would lead to a much better bound than (\ref{bdomega}), a contradiction with \cite[Theorem 10]{Bar05}).


\subsubsection{From convex bodies to ellipsoids} \label{rem:trois}
If Problem \ref{T} has a positive answer, then Corollary \ref{cor:large ellipsoid} would immediately follow from the beautiful and highly nontrivial result \cite[Theorem 2.11.9]{Tal21}. That result implies that given a convex body $K$ in $\R^n$ with $\gamma_n(K)\geq \frac{1}{2}$, there is a sequence of halfspaces $H_k=\{y\in \R^n; \langle x_k,y\rangle \geq 1\}$ such that for some universal $c_0\geq 1$,
$$\R^n\setminus c_0K \subset \bigcup_{k\geq 1}H_k \quad \text{and}\quad \|x_k\| \leq \frac{1}{\sqrt{\log (k+1)}}.$$

\section{Appendix}

\begin{proof}[Proof of Lemma \ref{corrrect}]
The proof argument closely follows \cite[Lemma 2.11]{Tal95}. 
Clearly, it is enough to show the statement when $n$ is an even integer, which we will assume below.
We want to estimate the sum over $i=1,...,n$ of 
\begin{align}\label{exi2}
\E[|X^*_i - r_i|^2] = 2\int_0^\infty \mathbb{P}[|X^*_i -r_i|>t]\, t\, dt.
\end{align} By symmetry of the question, it is enough to estimate the sum over $i=1,...,n/2$.
Fix $i\in\{1,...,n/2\}$ and $t>0$. Under these assumptions, $r_i\leq 0$.
Set $Z_{k}:= \mathbf{1}_{\{X_k < r_i-t\}}$. 
These random variables are independent. 
Note that $X^*_i-r_i < -t$ if and only if the number of elements in $\{k;\, X_k < r_i-t\}$ is at least $i$, namely
\begin{align}\label{ppz_k}
\mathbb{P}[X^*_i -r_i<-t]  = \mathbb{P}[\sum_{k=1}^n Z_{k} \geq i].\end{align}
Set $a_k:=\mathbb{P}[Z_{k}=1]$, and $a:= \sum_{k=1}^n a_k$. By our assumption on $\theta_i$ and the definition of $r_i$, 
$$a=n\Phi(r_i-t)  \in (0,i).$$
Write 
\begin{align}\label{pz_k}
\begin{split}
\mathbb{P}[\sum_{k=1}^n Z_{k} \geq i] & = \mathbb{P}[\sum_{k=1}^n (Z_{k} -a_k) \geq i-a]\\
&\leq \inf_{\lambda\geq 0} e^{-\lambda(i-a)} \Pi_{k=1}^n \E[e^{\lambda(Z_{k}-a_k)}].
\end{split}
\end{align} 
Next, for any $\lambda\geq 0$,
$$\E[e^{\lambda(Z_{k}-a_k)}] = a_ke^{\lambda (1-a_k)} +(1-a_k) e^{-a_k \lambda}\leq e^{a_k \lambda^2 e^\lambda}$$
(this can be checked by showing that for all $x\geq 0$ and $\lambda\geq 0$, $x e^\lambda +(1-x) \leq e^{x\lambda +x\lambda^2e^\lambda}$, which in turn can be proved by comparing the derivatives). Hence, by (\ref{pz_k}), we obtain
$$\mathbb{P}[\sum_{k=1}^n Z_{k} \geq i] \leq \inf_{\lambda\geq 0} 
e^{-\lambda(i-a) +a \lambda^2e^\lambda}.$$
If  $a\in [\frac{i}{7},i)$, we choose $\lambda=\frac{i-a}{6a}$ and get
\begin{align}\label{a>}
\mathbb{P}[\sum_{k=1}^n Z_{k} \geq i] \leq e^{-(i-a)^2/12a}
\end{align}
where we used $\lambda = \frac{i-a}{6a}\leq 1$ and $e^\lambda\leq 3$.\newline 
If  $a\in (0, \frac{i}{7})$, we choose $\lambda=\frac{1}{2}\log(\frac{i-a}{2a})$ and get
\begin{align}\label{a<}
\begin{split}
\mathbb{P}[\sum_{k=1}^n Z_{k} \geq i] & \leq 
e^{-\frac{(i-a)}{4}\log(\frac{i-a}{2a})}
\\
& = (\frac{2a}{i-a})^{\frac{i-a}{4}}\leq(\frac{2a}{i-a})^{\frac{i}{8}}  \leq (\frac{4a}{i})^{\frac{i}{8}}
\end{split}
\end{align}
where in the first line we used that $\lambda^2 e^\lambda \leq \lambda e^{2\lambda}$ so that $e^{-\lambda(i-a) +a \lambda^2e^\lambda} \leq e^{\lambda[-(i-a) + a e^{2\lambda}]}$, and in the second line we used that $\frac{2a}{i-a}\leq 1$, and $\frac{i-a}{4} \geq \frac{i}{8}$, and $\frac{2a}{i-a}\leq \frac{4a}{i}$. \newline
Set 
$$\xi_i(a):= e^{-\frac{(i-a)^2}{12a}} \quad \text{ if $a\in [\frac{i}{7},i)$},$$
$$\xi_i(a):= (\frac{4a}{i})^{\frac{i}{8}} \quad  \text{ if $a\in (0, \frac{i}{7})$}.$$
Combining (\ref{ppz_k}), (\ref{a>}) and (\ref{a<}), we obtain
\begin{align}\label{integral}
\int_0^\infty \mathbb{P}[X^*_i -r_i < -t]\, t\, dt \leq \int_0^\infty \xi_i(n\Phi(r_i-t)) \, t dt.
\end{align}

The estimate of the integral (\ref{integral}) is where there is an issue with the proof of \cite[Lemma 2.11]{Tal95}.
We will use the following basic inequalities:
\begin{align}\label{mills}
\text{for any $x\leq 0$,}\quad  \frac{|x|^2}{1+|x|^2}\Phi'(x)  \leq |x|\Phi(x) \leq \Phi'(x)  \quad (\text{Mills ratio inequalities})
\end{align}
\begin{align}\label{phibasic}
\text{for any $x\leq 0$, $t>0$,}\quad \Phi(x-t) \leq e^{-c(1+|x|)t}\Phi(x)\quad \text{where $c>0$ is universal.}
\end{align}
To check (\ref{phibasic}),  recall that the Gaussian cumulative distribution function $\Phi$ is log-concave, so $\frac{\Phi'(u)}{\Phi(u)}$ is non-increasing in $u$, and thus  $\frac{\Phi(x-t)}{\Phi(x)} =\exp(-\int_{x-t}^x \frac{\Phi'(u)}{\Phi(u)}du)\leq \exp(-t \frac{\Phi'(x)}{\Phi(x)}).$ 
Since  $\frac{\Phi'(x)}{\Phi(x)}\geq \frac{\Phi'(0)}{\Phi(0)} = \sqrt{\frac{2}{\pi}}$, and since we have $\frac{\Phi'(x)}{\Phi(x)} \geq |x|$ by the lower bound in the Mills ratio inequalities, we further deduce that $\frac{\Phi'(x)}{\Phi(x)}\geq c(1+|x|)$ for some universal $c>0$, which gives (\ref{phibasic}).
Next, recall that since $i\leq n/2$, we have $r_i\leq 0$. Besides, by the Mills ratio inequalities (\ref{mills}), one checks that
\begin{align}\label{r_iestimate}
(1+|r_i|)^2 \geq \frac{1}{10} (\log\frac{n}{i}+1).
\end{align}
By (\ref{phibasic}) and the definition of $r_i$, 
\begin{align}\label{phi--}
n\Phi(r_i-t)\leq i e^{-c(1+|r_i|)t} .
\end{align}
Define $t_i>0$ to be the unique number such that $n\Phi(r_i-t_i) =\frac{i}{7}$. In particular, $0< t_i\leq \frac{\log 7}{c}$ by (\ref{phi--}), and 
\begin{align}\label{mphi><}
\begin{split}
n\Phi(r_i-t) \in [\frac{i}{7},i)\quad \text{if $t\in (0,t_i]$,}\\
 n\Phi(r_i-t) \in (0,\frac{i}{7}) \quad \text{if $t\in (t_i,\infty)$}.
\end{split}
\end{align}
Besides, note that both functions $a\mapsto e^{-\frac{(i-a)^2}{12a}}$ and $a\mapsto (\frac{4a}{i})^{\frac{i}{8}}$ are non-decreasing in $a$. So by (\ref{mphi><}), (\ref{phi--}) and (\ref{phibasic}),
\begin{align}\label{xi_i><}
\begin{split}
 \text{if $t\in (0,t_i]$,}\quad \xi_i(n\Phi(r_i-t)) & = \exp({-\frac{(i-n\Phi(r_i-t))^2}{12n\Phi(r_i-t)}}) \leq \exp({-\frac{i^2 (1-e^{-c(1+|r_i|)t})^2}{12 i e^{-c(1+|r_i|)t}}})  \\
& \leq \exp({- \frac{i (1-e^{-c(1+|r_i|)t})^2}{12}}) \leq \exp({- \frac{i c^2(1+|r_i|)^2t^2}{12}}),
\\
 \text{if $t\in (t_i,\infty)$,} \quad \xi_i(n\Phi(r_i-t)) & = (\frac{4n\Phi(r_i-t)}{i})^{\frac{i}{8}} = (\frac{4n\Phi(r_i-t_i - (t-t_i))}{i})^{\frac{i}{8}}
 \\
& \leq (\frac{4n\Phi(r_i-t_i)e^{-c(1+|r_i|)(t-t_i)}}{i})^{\frac{i}{8}} \leq  
(\frac{4}{7}\exp[{-c(t-t_i)}])^{\frac{i}{8}}.
\end{split}
\end{align}
Thanks to (\ref{xi_i><}) and since $t_i\leq \frac{\log 7}{c}$, we can now estimate the integral (\ref{integral}):
\begin{align}\label{estimate integral}
\begin{split}
\int_0^\infty \xi_i(n\Phi(r_i-t)) \, t dt & = \int_0^{t_i} \xi_i(n\Phi(r_i-t)) \, t dt + \int_{t_i}^\infty \xi_i(n\Phi(r_i-t)) \, t dt  \\
& \leq \int_0^{t_i} \exp({- \frac{i c^2(1+|r_i|)^2t^2}{12}}) \, t dt  +\int_{t_i}^\infty (\frac{4}{7}\exp[{-c(t-t_i)}])^{\frac{i}{8}}\, t dt \\
& = \frac{6}{ic^2(1+|r_i|)^2}(1-\exp[-\frac{ic^2(1+|r_i|)^2t_i^2}{12}]) +\int_{t_i}^\infty
(\frac{4}{7})^{\frac{i}{8}} \exp[-\frac{ci}{8}(t-t_i)]
 \, t dt\\
& \leq   \frac{6}{ic^2(1+|r_i|)^2} + (\frac{4}{7})^{\frac{i}{8}} 
\int_{0}^\infty \exp({-\frac{c}{8}s})\, (s+\frac{\log 7}{c}) ds.
\end{split}
\end{align}
Finally we can bound (\ref{exi2}): summing the integrals (\ref{xi_i><}) over all $i=1,...,\frac{n}{2}$, we obtain using (\ref{estimate integral}) that for some universal constants $C', C''>0$, 
\begin{align*}
\begin{split}
\sum_{i=1}^{n/2} \E[|X^*_i - r_i|^2]
& = 2 \sum_{i=1}^{n/2} \int_0^\infty \mathbb{P}[X^*_i -r_i < -t]\, t\, dt \\
& \leq 2 \sum_{i=1}^{n/2} \int_0^\infty \xi_i(n\Phi(r_i-t)) \, t dt \\
&\leq C'\sum_{i=1}^{n/2}\frac{1}{i (1+|r_i|)^2} +C'\\ 
& \leq  C'\sum_{i=1}^{n/2}\frac{10}{i (\log\frac{n}{i}+1)} +C' \quad \quad \text{(use (\ref{r_iestimate})}\\
& \leq C''(\log \log n +1).
\end{split}
\end{align*}
 This is enough to conclude the proof.
\end{proof}


\bibliographystyle{alpha}
\bibliography{biblio_26_01_13}

\end{document}